\documentclass[reqno]{amsart}

\usepackage{appendix}
\usepackage{enumitem}
\usepackage{array}
\usepackage{color}
\usepackage{amsmath}
\usepackage{amsthm}
\numberwithin{equation}{section}

\usepackage{graphicx}

\newtheorem{thm}{Theorem}[section]

\newtheorem{lem}[thm]{Lemma}

\newtheorem{rem}{Remark}[section]

\newcommand{\E}{\mathcal{E}}
\newcommand{\T}{\mathbb{T}}
\newcommand{\C}{\mathcal{C}}
\newcommand{\LL}{\mathcal{L}}
\newcommand{\J}{\mathcal{J}}
\newcommand{\W}{\operatorname{W}}

\DeclareMathOperator*{\arginf}{arg\,inf}
\DeclareMathOperator*{\argmin}{arg\,min}

\numberwithin{figure}{section}

\title[Positivity-preserving scheme for quantum diffusion equation]{A positivity-preserving and energy stable scheme for a quantum diffusion equation}

\author{Xiaokai Huo}
\address{Computer, Electrical and Mathematical Science and Engineering Division,
King Abdullah University of Science and Technology (KAUST),Thuwal 23955-6900, Saudi Arabia}\email{xiaokai.huo@kaust.edu.sa}

\author{Hailiang Liu}
\address{Iowa State University, Department of Mathematics, Ames, IA 50011}
\email{hliu@iastate.edu}

\begin{document}
\begin{abstract}
    We propose a new fully-discretized finite difference scheme for a quantum diffusion equation, in both one and two dimensions. This is the first fully-discretized scheme with proven positivity-preserving and energy stable properties using only standard finite difference discretization. The difficulty in proving the positivity-preserving property lies in the lack of a maximum principle for fourth order PDEs. To overcome this difficulty, we reformulate the scheme as an optimization problem based on variational structure and use the singular nature of the energy functional near the boundary values to exclude the possibility of non-positive solutions. The scheme is also shown to be mass conservative and consistent.
\end{abstract}

\maketitle



\section{Introduction}
Nonlinear diffusion equations of fourth and higher order have since long been of interest in various fields of mathematical physics with diverse applications.  However, few mathematical results and numerical tools are available for higher-order equations when compared to the theory of second-order diffusion equations. One most important property for many higher-order diffusion equations is the positiveness of their solutions. 
However, the positiveness of solutions is rather subtle to prove due to the lack of a maximum principle.   The objective of this paper is to construct a novel numerical scheme for higher-order diffusion equations using a standard finite difference discretization, yet with proven positivity of numerical solutions. 

Here  we focus on the quantum diffusion equation on the torus $\mathbb{T}^d$ ($d \ge 1$ be any positive integer) with periodic boundary conditions:
\begin{align}\label{eq:DLSSnd}
  \partial_t u = - 2\nabla \cdot \left(u \nabla \frac{\Delta \sqrt{u} }{\sqrt{u}}\right), \quad u(\cdot, 0)=u_0 > 0. 
\end{align}
Here $u = u(x,t) \in \mathbb{R}^+$ is a scalar unknown function.

Equation \eqref{eq:DLSSnd} appears in several physical systems. In the one dimensional case, it is sometimes called the 
Derrida-Lebowitz-Speer-Spohn equation and was first introduced in \cite{derrida1991fluctuations} 
in the context of spin systems.  
The multi-dimensional equation often 
appears in the context of semi-conductor modeling, and can be viewed as the evolution of the density of electrons with vanishing temperature of the simplified quantum drift-diffusion model \cite{An87, GMSU95}:
$$
\partial_t u = {\rm div}(T\nabla u+u\nabla V), \quad  V=V_e-\frac{\epsilon^2}{6} \frac{\Delta \sqrt{u}}{\sqrt{u}}.
$$
Here $T>0$ is the temperature, $\epsilon$ the Planck constant, and $V$ the potential felt by the electrons, which splits into the classical electric potential  $V_e$ and the Bohm potential, 
describing quantum effects. 
The equation can also be found from quantum hydrodynamics as the high friction limit of some quantum hydrodynamic equations, see \cite{GLT17,GT17,LT17}.

\subsection{Related work} 
Positivity of solutions plays an important role in the analysis of equation \eqref{eq:DLSSnd}. 
In the one dimensional case, the local-in-time existence and uniqueness 
were proved in \cite{BLS94}. 
For initial data $u_0(x)>0$ almost everywhere, such solution remains positive until the maximal existence time $T^*$, in the sense that $\lim_{t\to (T^*)^-} u(\cdot,t)$ vanishes at some point.
However, it is not known whether $T^*=\infty$. 
The only known global existence results are for weak solutions, proved in \cite{JP00} for the one dimensional case and in \cite{GST08,JM08} for higher dimensions.
Further theoretical studies of this equation can be found in  \cite{derrida1991fluctuations,CCT03,JT03}.

Equation \eqref{eq:DLSSnd} can be formally put as a gradient flow of form 
\begin{align}\label{eq:dlss-general}
\partial_t u =\nabla \cdot \left( u \nabla \left( 
\frac{\delta F}{\delta u}
\right) \right),
\end{align}
where the \emph{energy functional}  $F:=F(u)$ is given by
\begin{align}\label{eq:Fdef}
F(u)=\frac{1}{2}\int_{\mathbb{T}^d} \frac{|\nabla u|^2}{u}dx.
\end{align}
Note that this functional is also called \emph{Fisher information} in the literature. A rigorous justfication of 
the gradient flow structure of this equation was given in \cite{GST08} using 
the Wasserstein distance 
\begin{align}\label{eq:wassersteinmetric}
  \W_2^2(u^0,u^1) = \inf_{\gamma \in \Pi(u^0,u^1)} \int_{\mathbb{T}^d \times \mathbb{T}^d} |x-y|^2 d\gamma(x,y), 
\end{align}
where $\Pi(u^0, u^1)$ is the set of all probability measures on $\mathbb{T}^d \times \mathbb{T}^d$ with first marginal $u^0$ and second marginal $u^1$, and the symbol $|\cdot|$ denotes the usual Euclidean norm in $\mathbb{T}^d$. 

Positivity of numerical solutions is highly desirable for any numerical method to solve the equation \eqref{eq:DLSSnd}. 
One way to obtain positivity is to introduce auxiliary variables and enforce the solution positivity accordingly. For example, one can introduce $v=\log u$  \cite{CJT03,JP01} or $v = \sqrt{u}$ \cite{BEJ14}  
to ensure a positive  solution. 
Another way to obtain positivity is to use the Wasserstein gradient flow structure. 
In the one dimensional case, the Wasserstein distance \eqref{eq:wassersteinmetric} corresponds to  
the $L^2$ distance between Lagrangian maps.  Positivity of $u$ then follows from 
the monotonicity of the Lagrangian map \cite{MO17}. 
One can also utilize the Eulerian formulation of the Wasserstein gradient flow structure. More precisely, by the Benamou-Brenier formulation of the Wasserstein distance \cite{BB00}, which is defined by 
\begin{align*}
  \W^2_2(u^0,u^1):=\inf_{u, m} \bigg\{& \int_0^1 \int_{\mathbb{T}^d} \frac{m^2}{u} dxdt, \\
  &\text{s.t. } \partial_t u + \nabla\cdot m=0,~~ u(0,x)=u^0(x),~~ u(1,x)=u^1(x) \bigg\},
\end{align*}
the scheme proposed in \cite{LLW} for \eqref{eq:DLSSnd} takes  the form  
\begin{align} \label{eq:jko}
    u^{n+1} =\arginf_{u,m}\bigg\{ &\int_{\mathbb{T}^d} \frac{m^2}{2u}dx  + \Delta t F(u), \\
    &\text{s.t. } u - u^n(x) + \nabla \cdot m(x) =0\bigg\}.\nonumber
\end{align}
Here $\Delta t$ is the time step and $u^n$ is the solution at the $n$-th time step. 
Positivity of {$u^{n+1}$} was proved since the objective function in the above optimization problem becomes infinite when $u$ touches zero \cite{LLW}.

Some other numerical approximations  such as finite volume methods were also developed to solve this equation, see, e.g., \cite{MM16}. However, the existing positivity-preserving numerical methods for equation \eqref{eq:DLSSnd} do not seem to follow simply from a direct finite difference discretization. In \cite{MO17}, the authors even argued that there is no reason to expect the positivity-preserving property from a standard discretization approach. 

\subsection{Our contributions}
The main objectives of this paper are to present a novel numerical scheme  to approximate  \eqref{eq:DLSSnd} using a standard finite difference discretization and prove how positivity of numerical solutions can be obtained from such a discretization. 

We begin with the semi-discretized scheme
\begin{align}\label{eq:sch}
    \frac{u^{n+1}-u^n}{\Delta t} = \nabla \cdot (u^n \nabla H^{n+1}), \quad H^{n+1} = \delta_{u} F(u^{n+1}),
\end{align}
where $H^{n+1}$ is obtained by the variational derivative of energy functional $F=F(u)$. This scheme is formally equivalent to the optimization problem 
\begin{align}\label{eq:optimiz}
    u^{n+1} =  \arginf_{u} \left\{ \frac{1}{2} \|u-u^n\|_{\LL^{-1}_{u^n}}^2 + \Delta t F(u)\right\}.
\end{align}
Here  with $\LL_{u^n}\phi: =-\nabla \cdot (u^n \nabla \phi)$ we define 
\[\|f\|_{\LL^{-1}_{u^n}}^2 := \left\{ \int_{\mathbb{T}^d} u^n|\nabla \phi| ^2 dx: 
f=\LL_{u^n}\phi
\right\}.
\]
The above optimization problem can be rewritten as 
\begin{align} \label{eq:jko+}
    u^{n+1} =\arginf_{u,m} \bigg\{& \int_{\mathbb{T}^d} \frac{|m|^2}{2u^n }dx +\Delta t F(u) \\
    &\text{s.t. } u - u^n(x) + \nabla \cdot m =0\bigg\}.\nonumber
\end{align}
The optimization problem \eqref{eq:jko+} is similar in structure to the scheme \eqref{eq:jko}. Note that the optimal conditions for this optimization problem lead to precisely the scheme (\ref{eq:sch}), while  the optimal conditions for \eqref{eq:jko} link only to a first order approximation of (\ref{eq:sch}) or (\ref{eq:sc1}), see \cite{LLW}. 

The singular nature of $F(u)$ at $u=0$ may prevent the minimizer of \eqref{eq:jko+} from reaching zero,  
 so
it is expected that solution positivity can be deduced at least after some spatial discretization.  In this work we will prove solution positivity 
 for the fully-discretized numerical scheme of form 
\begin{align}\label{eq:fulld}
  \frac{u^{n+1}-u^n}{\Delta t} = d_h \cdot(\hat{u}^nD_h H^{n+1}), \quad H^{n+1} = \frac{\partial F_h}{\partial u}(u^{n+1}).
\end{align}
Here $\hat{u}$ is defined to be an average value of $u$ on the adjacent grid points given by \eqref{eq:Axax} and \eqref{eq:Axax2d}, and  $d_h,D_h$ are finite difference operators, given by \eqref{eq:Dhdh} and \eqref{eq:Dh2d}-\eqref{eq:dh2d}, both for one and two dimensions, respectively. The function $F_h$ is a discretized version of the functional \eqref{eq:Fdef}, given by \eqref{eq:Fhdef} and \eqref{eq:fh2d} for the one and two dimensions, respectively. 

Our approach in proving the positivity-preserving property using optimization formulations 
was motivated by the works \cite{chen2017positivity} and \cite{dong2018positivity}, where the authors used the finite difference discretization to study the Cahn-Hillard equation
\[\partial_t \phi = \nabla \cdot (M(\phi) \nabla \mu), \; \mu = \delta_\phi E\]
with the energy functional $E$ given by, for example
\begin{align*}
    E(\phi) = \int_{\Omega} \left( (1+\phi)\log \phi + (1-\phi)\log (1-\phi) - \frac{\theta_0}{2}\phi^2 + \frac{\varepsilon^2}{2} |\nabla \phi|^2 \right) dx
\end{align*}
in \cite{chen2017positivity}.
Our problem here differs from the Cahn-Hillard equation in that the functional includes a higher order term which develops singularity at zero, and the possibility of $\nabla u$ and $u$ being zero simultaneously also poses additional challenges. 

The main results of this paper include a new fully discretized finite difference scheme \eqref{eq:fulld} to solve \eqref{eq:DLSSnd}, in both one and higher dimensions, and rigorous proofs of scheme properties such as positivity-preserving, mass conservation, energy stability, and scheme consistency. 
One most remarkable contribution of this work is to show that standard finite difference discretization of the quantum diffusion equation \eqref{eq:DLSSnd} can meet the requirement of both solution positivity and energy dissipation simultaneously. 

The paper is organized as follows. In the next section we present four semi-discrete schemes with only time discretization for equation \eqref{eq:DLSSnd} and briefly discuss their properties. Section 3 is devoted to the fully discretized scheme in the one dimensional case, where the scheme is shown to be consistent. The optimization formulation is given in Section 4, where the scheme \eqref{eq:fulld} is shown to be positivity-preserving, energy dissipative and mass conservative. In Section 5, we present and analyze the scheme in higher dimensions, taking the two dimensional case as an example. Finally,  some numerical examples are presented in Section 6.

\subsection*{Notations} We use $L^\infty$ to denote the space of bounded sequences. We use bold symbol $\mathbf{f} = (f^x,f^y)$ to denote vector in $\mathbb{R}^2$. We take $h$ to be the mesh size and $\Delta t$ to be the time step in discretization.

\section{Time discretization}
In this section we present four semi-discrete schemes with only time discretization, and comment on both pros and cons of each scheme. 
\subsection{Explicit scheme}
The simplest scheme is the explicit scheme:
\begin{align}\label{eq:explicit}
  \frac{u^{n+1}-u^n}{\Delta t} = -2\nabla \cdot \left(u^n \nabla \frac{\Delta \sqrt{u^{n}}}{\sqrt{u^n}}\right).
\end{align}
This scheme is easy to implement. However, the sign of the right hand side is not certain and the positivity of solutions is difficult to show. The scheme can also be written in the following form
\begin{align}\label{eq:22} 
    \frac{u^{n+1}-u^n}{\Delta t} = \nabla \cdot \left(u^n \nabla \frac{\delta F}{\delta u}(u^{n})\right).
\end{align}
Assume $u^n>0$ and $u^{n+1}>0$, we calculate the relative energy defined by 
\begin{align*}
  F(u^{n+1}|u^n):=& F(u^{n+1}) - F(u^n) - \int_{\mathbb{T}^d} \frac{\delta F}{\delta u}(u^n) (u^{n+1}-u^n)dx \\
  =& \frac{1}{2}\int_{\mathbb{T}^d} \frac{|\nabla u^{n+1}|^2}{u^{n+1}} dx - \frac{1}{2}\int_{\mathbb{T}^d} \frac{|\nabla u^n|^2}{u^n} dx + \frac{1}{2} \int_{\mathbb{T}^d} \frac{|\nabla u^n|^2}{(u^n)^2}(u^{n+1}-u^n)dx \\
  &- \frac{1}{2} \int_{\mathbb{T}^d} \frac{2\nabla u^n\cdot(\nabla u^{n+1}-\nabla u^n)}{u^n} dx \\
  =&\frac{1}{2} \int_{\mathbb{T}^d}u^{n+1} \left(\frac{\nabla u^{n+1}}{u^{n+1}} - \frac{\nabla u^n}{u^n}\right)^2dx.
\end{align*}
From equation \eqref{eq:22}, we have
 \begin{align*}
  \int_{\mathbb{T}^d}\frac{\delta F}{\delta u}(u^{n})(u^{n+1}-u^{n}) dx =& \Delta t \int_{\mathbb{T}^d}\frac{\delta F}{\delta u}(u^{n}) \nabla \cdot\left(u^{n}\nabla\frac{\delta F}{\delta u}(u^{n})\right) dx \\
  =& -\Delta t \int_{\mathbb{T}^d}u^{n}\left|\nabla\frac{\delta F}{\delta u}(u^{n})\right|^2 dx  \le 0,
  \end{align*}
  So the energy difference at each time step will change by
  \begin{align*}
    F(u^{n+ 1}) - F(u^n) = \frac{1}{2} \int_{\mathbb{T}^d}u^{n+1} \left(\frac{\nabla u^{n+1}}{u^{n+1}} - \frac{\nabla u^n}{u^n}\right)^2dx - \Delta t \int_{\mathbb{T}^d}u^{n}\left|\nabla\frac{\delta F}{\delta u}(u^{n})\right|^2 dx.
  \end{align*}
The above relation shows that the second term on the right hand side will be negative while the first term be positive.  The right hand side is not guaranteed to have a definite sign unless $\Delta t$ is very small. 
 So for the explicit scheme, the positivity-preserving and energy dissipation  properties are not guaranteed.
  
\subsection{Fully implicit scheme}
For stability reasons, we prefer to use an implicit scheme instead of an explicit one. The implicit scheme for the equation \eqref{eq:DLSSnd} reads as
\begin{align}\label{eq:sc1}
  \frac{u^{n+1}-u^n}{\Delta t} =  - 2\nabla \cdot \left(u^{n+1} \nabla \frac{\Delta \sqrt{u^{n+1}} }{\sqrt{u^{n+1}}}\right).
\end{align}
Such a scheme was studied in \cite{BEJ14}, where the existence of weak solutions and the dissipation of Fisher information were shown. In other words,  assuming $u^n>0$ and $u^{n+1}>0$, the scheme dissipates the energy.  
For the convenience of reference, we present this result and a proof as follows. 
\begin{lem}\label{lm1}
  Suppose $u^n, u^{n+1}>0$ everywhere in $\mathbb{T}^d$, then 
  \begin{align}\label{eq:f1}
    F(u^{n+1})\le F(u^n)
  \end{align}
  for any $\Delta t>0$.
\end{lem}
\begin{proof}
  We calculate the relative energy $F(u^{n}|u^{n+1})$ by
  \begin{align}\label{eq:Frelative}
  F(u^{n}|u^{n+1})& = F(u^{n}) - F(u^{n+1}) - \int_{\mathbb{T}^d}\frac{\delta F}{\delta u}(u^{n+1})(u^{n}-u^{n+1}) dx \nonumber\\
  =& \frac{1}{2} \int_{\mathbb{T}^d} \frac{|\nabla u^{n}|^2}{u^{n}} dx - \frac{1}{2} \int_{\mathbb{T}^d} \frac{|\nabla u^{n+1}|^2}{u^{n+1}} dx + \frac{1}{2} \int_{\mathbb{T}^d} \frac{|\nabla u^{n+1}|^2}{(u^{n+1})^2} (u^{n}-u^{n+1})  dx \nonumber\\
    &-\frac{1}{2} \int_{\mathbb{T}^d} \frac{2 \nabla u^{n+1} \cdot (\nabla u^n-\nabla u^{n+1})}{u^{n+1}}dx \nonumber\\
    =& \frac{1}{2} \int_{\mathbb{T}^d} u^n \left(\frac{\nabla u^n}{u^n} - \frac{\nabla u^{n+1}}{u^{n+1}}\right)^2 dx \ge 0.
  \end{align}
  We learned from the gradient structure that equation \eqref{eq:sc1} can be written as
  \[\frac{u^{n+1}-u^n}{\Delta t} = \nabla \cdot \left(u^{n+1}\nabla\frac{\delta F}{\delta u}(u^{n+1})\right),\]
  so 
  \begin{align}\label{eq:deltaFu}
  \int_{\mathbb{T}^d}\frac{\delta F}{\delta u}(u^{n+1})(u^{n}-u^{n+1}) dx =& - \Delta t \int_{\mathbb{T}^d}\frac{\delta F}{\delta u}(u^{n+1}) \nabla \cdot\left(u^{n+1}\nabla\frac{\delta F}{\delta u}(u^{n+1})\right) dx \nonumber\\
  =& \Delta t \int_{\mathbb{T}^d}u^{n+1}\left|\nabla\frac{\delta F}{\delta u}(u^{n+1})\right|^2 dx  \ge 0,
  \end{align}
  due to the positivity of $u^{n+1}$.  Substituting it into \eqref{eq:Frelative} gives
  \begin{align}\label{eq:fle}
     F(u^n) - F(u^{n+1}) \ge F(u^{n}|u^{n+1}) \ge 0 .
  \end{align}
\end{proof}
Notice that the positivity property is crucial in establishing the energy stability, but it seems difficult to directly prove such a property.  Given $u^n$, the equation for $u^{n+1}$ is a fourth order nonlinear elliptic equation, the positivity of $u^{n+1}$ does not seem to be derivable from a maximum principle. 

\subsection{A positivity-preserving scheme} 
Drawing ideas from \cite{LY12,  LM19, LM19+} in the design of unconditionally positive schemes for 
second-order Fokker-Planck equations, we set 
$$
H=-\frac{2\Delta \sqrt{u}}{\sqrt{u}}, \quad M=e^{\ln u +H}, 
$$
equation \eqref{eq:DLSSnd} can be rewritten in the form
\begin{align*}
  \partial_t u  = 
  \nabla \cdot \left(M\nabla \left(\frac{u}{M}\right)\right),
\end{align*}
which can be approximated by 
\begin{align}\label{eq:sc2}
\frac{u^{n+1}-u^n}{\Delta t}=\nabla  \cdot \left( M^n \nabla  \left( 
\frac{u^{n+1}}{M^n}
\right)\right).
\end{align}
This scheme is linear in $u^{n+1}$ and hence easy to code. Also it is unconditionally positivity-preserving.
\begin{lem} If $u^n>0$, then 
$$
u^{n+1}>0
$$
for any $\Delta t>0$.
\end{lem}
\begin{proof}Set $G^{n+1}=u^{n+1}/M^n$ so that 
$$
G^{n+1}M^n -\Delta t \nabla  \cdot (M^n \nabla  G^{n+1}))=u^n.
$$
Note that both existence and regularity of the solution $G^{n+1}$  are ensured by the classical elliptic theory. Assume $G^{n+1}$ achieves its minimum at $x_0$, then $\nabla  G^{n+1}(x_0)=0$, and $\Delta G^{n+1}(x_0) \geq 0$.
From the equation when evaluated at $x_0$ it follows that 
$$
G^{n+1}(x_0)M^n(x_0)=\Delta t M^n(x_0) \Delta G^{n+1}(x_0) +u^n(x_0) \geq u^n(x_0)>0.
$$
Hence $G^{n+1}>0$ for all $x\in \mathbb{T}^d$, so is $u^{n+1}$.  
\end{proof}
It seems less obvious to verify the energy dissipation property.
\subsection{An explicit-implicit scheme} 
The fourth scheme is  
\begin{align}\label{eq:sc3}
\frac{u^{n+1}-u^n}{\Delta t}=\nabla  \cdot \left( u^{n} \nabla  H^{n+1}\right),
\end{align}
where $$H^{n+1}= \frac{\delta F}{\delta u}(u^{n+1}) = -\frac{2\Delta \sqrt{u^{n+1}}}{\sqrt{u^{n+1}}}.
$$
One can view  this scheme as an intermediate one between \eqref{eq:sc1} and \eqref{eq:explicit} or \eqref{eq:sc2}.
The scheme is energy stable assuming $u^n$ and $u^{n+1}$ are both positive. 
\begin{lem}\label{lem:23}
  Suppose $u^n,u^{n+1} > 0$, then 
  \begin{align} \label{eq:lm23}
    F(u^{n+1}) \le F(u^n)
  \end{align}
  for any $\Delta t>0$.
\end{lem}
\begin{proof}
  The proof is similar to the proof of Lemma \ref{lm1}, except that  
  \begin{align}
  \int_{\mathbb{T}^d}\frac{\delta F}{\delta u}(u^{n+1})(u^{n}-u^{n+1}) dx =& - \Delta t \int_{\mathbb{T}^d}\frac{\delta F}{\delta u}(u^{n+1}) \nabla \cdot\left(u^{n}\nabla\frac{\delta F}{\delta u}(u^{n+1})\right) dx \nonumber\\
  =& \Delta t \int_{\mathbb{T}^d}u^{n}\left|\nabla\frac{\delta F}{\delta u}(u^{n+1})\right|^2 dx  \ge 0, \label{eq:211}
  \end{align}
  so \eqref{eq:fle} holds, i.e., \eqref{eq:lm23} holds.
\end{proof}
Notice that here the assumption $u^{n+1}> 0$ is also needed in getting the result, since the function $\delta F/\delta u(u^{n+1})(x)$ is undefined where $u^{n+1}(x)=0$.
In the next section, we shall further discretize \eqref{eq:sc3} in space and prove the positivity of numerical solutions in the subsequent sections. 

\section{The full discrete scheme in one dimension}
\subsection{Notations} \label{sec:not}
We follow the notations in \cite{wise2009energy}. First
we define the following two grids on the torus $\T=[0,L]$ with mesh size $h=L/N$, where $N$ is the number of mesh intervals:
\begin{align}\label{eq:cedef}
\C:= \left\{h, 2h,\ldots,L\right\},\quad \E:=\left\{\frac{h}{2},\frac{3h}{2},\ldots,\frac{2N-1}{2}h\right\}.   
\end{align}
We treat $\C$ and $\E$ as periodic. For example, we write $x_i$ as the $i$-th element in $\C$, then $x_N=L$ and $x_{N+1}=x_1=h$. The elements in $\E$ then can be written as $x_{i+\frac{1}{2}}$ with $i\in\mathbb{Z}$. We define the discrete $N$-periodic function space as
\[\C_{\rm per}:=\{f:\C \to \mathbb{R}\},\quad \E_{\rm per}:=\{f:\E\to\mathbb{R}\}.\]
Here we call $\C_{\rm per}$ the space of \emph{cell centered functions} and $\E_{\rm per}$ the space of \emph{edge centered functions}. We also define a subspace of $\C_{\rm per}$ by
\[\mathring{\C}_{\rm per}:=\left\{f: f\in \C_{\rm per},\,\sum_{i=1}^N f_i = 0\right\}.\]
The discrete gradient $D_h$ and $d_h $ are defined to be
\begin{align}\label{eq:Dhdh}
  (D_h f)_{i+\frac{1}{2}}:= \frac{f_{i+1}-f_i}{h}, \quad (d_h f)_i:=\frac{f_{i+\frac{1}{2}}-f_{i-\frac{1}{2}}}{h}.
\end{align}
We define the average of the function values of nearby points by
\begin{align}\label{eq:Axax}
\hat{f}_{i+\frac{1}{2}} = \frac{f_i + f_{i+1}}{2}, \text{ if } f \in \mathcal{C}_{\rm per}, \quad\text{and}\quad \hat{f}_{i} = \frac{f_{i+\frac{1}{2}}+f_{i-\frac{1}{2}}}{2}, \text{ if } f \in \mathcal{E}_{\rm per}.
\end{align}
The inner products are defined by
\[\langle f,g \rangle := h \sum_{i=1}^N f_i g_i,\, \forall f,g \in \C_{\rm per}, \quad [f,g]:=\frac12 h\sum_{i=1}^N (f_{i-\frac12} g_{i-\frac12}+f_{i+\frac12}g_{i+\frac12}), \, 
\forall f,g \in \E_{\rm per}.\]
The corresponding norms in each space are defined by
\[\|f\|_{\C_{\rm per}}=\sqrt{h\sum_{i=1}^N f_i^2},\, \text{ for } f \in \C_{\rm per}, \quad  \|f\|_{\E_{\rm per}}=\sqrt{h\sum_{i=1}^N \frac{1}{2}(f_{i+\frac{1}{2}}^2+f_{i-\frac12}^2)},\, \text{ for } f \in \E_{\rm per}.\]
Suppose $f,g\in \C_{\rm per}$ and $\phi \in \E_{\rm per}$, the following summation-by-parts formulas hold:
\begin{align}
  \langle f,d_h \phi \rangle = -[D_h f,\phi],\quad 
  \langle f, d_h(\phi D_h g)\rangle 
  = -[D_h f,\phi D_h g].
\end{align}
We next introduce a weighted norm for $ g\in \mathring{\C}_{\rm per}$ with positive weight 
$\phi \in \E_{\rm per}$. Let $f:= \LL_\phi^{-1} g \in \mathring{\C}_{\rm per}$ be determined by 
\begin{align} \label{eq:Ldef}
  \LL_\phi(f)=-d_h(\phi D_h f) = g,
 \end{align} 
 then the norm of $g$ on $\mathring{\C}_{\rm per}$ is defined by 
 \begin{align}\label{eq:phinorm}
  \|g\|_{\LL_\phi^{-1}} = \sqrt{[\phi D_h f,D_h f]}. 
\end{align}
The above norm can also be induced by the bilinear form on $\mathring{\C}_{\rm per}$:
 \[\langle g_1, g_2 \rangle_{\LL_\phi^{-1}}:=[\phi D_h f_1,D_h f_2],\, \forall g_1, g_2 \in \mathring{\C}_{\rm per},
 \]
with  the following property: 
\[
\langle g_1,g_2 \rangle_{\LL_\phi^{-1}} = \langle g_1, \LL_\phi^{-1} g_2 \rangle = \langle \LL_\phi^{-1} g_1,g_2\rangle.
\]

\subsection{The scheme} 
We proceed to study the time discretization scheme \eqref{eq:sc3}.  We adopt the following fully discrete 
scheme as
\begin{align}\label{eq:timed}
  \frac{u^{n+1} - u^{n}}{\Delta t} = d_h \left(\widehat{u}^n D_h H^{n+1}\right),
\end{align}
where $D_h,\,d_h$ are the difference operators defined in \eqref{eq:Dhdh}, $\hat{u}$ is the average operator defined in \eqref{eq:Axax}, and  $H^{n+1}$ is taken to be 
\begin{align}\label{eq:mudiscrete}
  H^{n+1}=& \frac{1}{h}\frac{\partial F_h}{\partial u}(u^{n+1}),
\end{align}
where the discretized energy functional 
\begin{align}\label{eq:Fhdef} 
  F_h(u):= 
  \frac{1}{h}\sum_{i=1}^N \frac{(u_{i+1}-u_i)^2}{2 u_i}.
\end{align}
The scheme can be written explicitly as 
\begin{align*}
  \frac{u^{n+1}_i-u^n_i}{\Delta t}
  = \frac{(u^n_{i+1}+u^n_i)(H^{n+1}_{i+1}-H^{n+1}_{i})-(u^n_{i}+u^n_{i-1})(H^{n+1}_{i}-H^{n+1}_{i-1})}{2h^2}
\end{align*}
with 
$$
 H^{n+1}_i = - \frac{1}{2h^2} \frac{(u_{i+1}^{n+1}-u_i^{n+1})^2}{(u_i^{n+1})^2} - \frac{1}{h^2}\left(\frac{u_{i+1}^{n+1}-u_i^{n+1}}{u_i^{n+1}} - \frac{u_{i}^{n+1}-u_{i-1}^{n+1}}{u_{i-1}^{n+1}}\right).
$$



\subsection{Consistency of the scheme}
 Here we show the consistency of the scheme. Suppose $u=u(x,t)$ is a smooth solution to the equation \eqref{eq:DLSSnd} and $u_j^n$ is a finite difference solution to \eqref{eq:timed}-\eqref{eq:mudiscrete}, we will show that the local truncation error defined by
\begin{align}\label{eq:tau}
  \tau_j^n := \frac{u (x_j,t_{n+1}) -u(x_j,t_n)}{\Delta t} -d_h\left(\hat{u}(x,t_n)D_h H^{n+1}(u(x,t_{n+1}))\right)_{x_j} 
\end{align}
converges to $0$ as $\Delta t,h \to 0$.

The local truncation error $\tau_j^n$ can be computed using Taylor's expansion. 
A careful calculation gives the following result, in which each term is evaluated at $(x_j,t_n)$: 
\begin{align*}
\tau_{j}^n =
&- h \left(-\frac{3 {u_x}^5}{u^4}+\frac{6 u_{xx}u_x^3}{u^3}-\frac{3 u_{xxx} u_x^2}{2 u^2}-\frac{3 u_{xx}^2 u_x}{2 u^2}\right)-k \bigg(\frac{3 u_t u_x^4}{u^4}-\frac{3 u_{xt} u_x^3}{u^3}\\
&+\frac{2 u_{xxt} u_x^2}{u^2}-\frac{4 u_t u_{xx} u_x^2}{u^3}+\frac{2 u_{xt} u_{xx} u_x}{u^2}+\frac{u_t u_{xxx} u_x}{u^2}-\frac{u_{xxxt} u_x}{u}\bigg)+ o(h+k).
\end{align*}
We have the following theorem. 
\begin{thm} 
Suppose that the solution $u=u(x,t)$ to the equation \eqref{eq:DLSSnd} is positive and smooth,  then the consistency error for the numerical scheme \eqref{eq:timed}-\eqref{eq:mudiscrete}, 
defined by \eqref{eq:tau}, satisfies
  \begin{align*}
    \|\tau^n\|_{L^\infty} \le C(h+k),
  \end{align*}
  where $C$ depends on $u$.
\end{thm}

\section{Solution properties via optimization formulation}
In this section, 
we show that the finite difference scheme can be reformulated as an optimization problem by using the gradient flow structure. Inspired by the formal equivalence of \eqref{eq:optimiz} with the scheme \eqref{eq:sch}, we will prove the equivalence after discretized in space. 
More precisely, we first establish the following theorem.
 \begin{thm}\label{thm}  Given ${u}^n \in \C_{\rm per}$ positive with $K=\sum_{i=1}^N h {u}_i^n>0$, there exists a constant $\delta_0>0$ such that
$u^{n+1}>0$ is the solution of the numerical scheme \eqref{eq:timed}-\eqref{eq:mudiscrete} if and only if it is the unique minimizer to the following optimization problem
  \begin{align}\label{eq:min}
  u^{n+1}=\argmin_{u\in \mathcal{A}_{h,\delta}}\left\{\J[u]=\frac{1}{2\Delta t}\|u-u^n\|^2_{\LL_{\hat{u}^n}^{-1}} + F_h(u)\right\}, 
  \end{align}
  over the set
   \[ \mathcal{A}_{h,\delta} =\left\{ u \in \C_{\rm per} \;: \; 
    u \ge \delta,
    \quad 
    h\sum_{i=1}^N u_i = K
    \right\},\]
   for  any $0<\delta \le \delta_0$. 
 Here $F_h(u)$ is defined in \eqref{eq:Fhdef} and the norm $\|\cdot\|_{\LL_{\hat{u}^n}^{-1}}$ is defined in \eqref{eq:phinorm}.
 \end{thm}

 Before proving this theorem, we first prepare a lemma, which is due to \cite{chen2017positivity} and holds for both one and higher dimensions. Here we present the result for the one dimension case. 
\begin{lem} \label{lm:phi}
 Suppose $\phi \in \E_{\rm per}$ has a positive minimum $\min \phi
 >0$. Let $g \in \mathring{\C}_{\rm per}$ be bounded in $L^\infty$, then it holds that
  \begin{align}\label{eq:lem1}
    \|\mathcal{L}_{\phi}^{-1}g\|_{L^\infty} \le \frac{C}{\min{\phi}} h^{-\frac{1}{2}}L \|g\|_{L^\infty},
  \end{align}
  where $\min \phi$ is the minimum of $\phi$ over the grid points and $C$ does not depend on $h$.
\end{lem}
\begin{proof}
We adapt the proof of Lemma 3.2 in \cite{chen2017positivity}. 
 By the definition \eqref{eq:phinorm},
\begin{align*}
   (\min \phi)\|D_h f\|_{L^2}^2  &\le  \left[\phi D_h f,D_h f \right]\\
  & = \|g\|_{\LL^{-1}_{\phi}}^2=
  \langle -d_h(\phi D_h f),f \rangle = \langle g,f\rangle \le \|g\|_{L^2}
 \|f\|_{L^2}.
 \end{align*}
Since 
$f\in \mathring{\C}_{\rm per}$,
the use of the discrete Poincar\'e inequality gives 
 \[\|f\|_{L^2} \le C \|D_h f\|_{L^2}, \]
which when inserted into the previous inequality leads to 
 \[\|D_h f\|_{L^2} \le  \frac{C}{\min \phi}\|g\|_{L^2}.\]
Using an inverse inequality leads to 
\[\|f\|_{L^\infty} \le C h^{-\frac{1}{2}} \|D_h f\|_{L^2} \le \frac{C}{\min \phi} h^{-\frac{1}{2}} \|g\|_{L^2} \le \frac{C}{\min \phi} h^{-\frac{1}{2}} L \|g\|_{L^\infty},\]
where $C>0$ is a constant indepedent of $h$.
\end{proof}
\begin{proof} The proof is divided into three steps: 

Step 1. We  prove the existence and uniqueness of the optimization problem \eqref{eq:min}
for any $\delta >0$. According to the definition \eqref{eq:phinorm},  we have  
 \begin{align*}
    \|u-u^n\|^2_{\LL_{\hat{u}^n}^{-1}}  = \left[\hat{u}^n D_h f,D_h f\right],
  \end{align*}
  where $f\in \mathring{\C}_{\rm per}$ solves 
  \begin{align*}
    u-u^n = \mathcal{L}_{\hat{u}^n} f = -d_h(\hat{u}^nD_hf).
  \end{align*}
Set $v = \hat{u}^n D_h f$, we get 
  \[u-u^n = -d_h(v), \quad \|u-u^n\|^2_{\LL_{\hat{u}^n}^{-1}} = \left[\frac{1}{\hat{u}^n} v,v\right].
  \]
  The optimization problem \eqref{eq:min} is equivalent to the following problem
  \begin{align}\label{eq:opconvex}
    u^{n+1} = \argmin_{(u,v) \in V_{h,\delta}} \left\{\mathcal{J}_1(u,v)=\frac{1}{2\Delta t}\left[\frac{1}{\hat{u}^n} v,v\right] + F_h(u) \right\},
  \end{align}
  with the set $V_{h,\delta}$ defined by
  \[V_{h,\delta}:=\left\{(u,v): u \in \mathcal{A}_{h,\delta},~~u-u^n+d_h(v)=0\right\}.\] 
  Notice that $V_{h,\delta}$ is a region given by linear constraints and thus is convex. The objective function $\mathcal{J}_1(u,v)$ is also convex due to the convexity of 
  \[\left[\frac{1}{\hat{u}^n} v,v\right]\]
  in $v$ and the convexity of $F_h(u)$ in $u$, which can be seen from that each term 
  \[\frac{1}{h} \frac{(u_{i+1}-u_i)^2}{2u_i}\]
  in the summation in \eqref{eq:Fhdef} is a convex function of $u_i$ and $u_{i+1}$. Therefore, the problem \eqref{eq:opconvex} is a convex optimization problem and thus has a unique solution. Since \eqref{eq:opconvex} and \eqref{eq:min} are equivalent, there also exists a unique solution to the optimization problem \eqref{eq:min} for any $\delta >0$. 

Step 2.  We show that $u^{n+1}>0$ is the solution to the numerical scheme \eqref{eq:timed}-\eqref{eq:mudiscrete}  if and only if it is the minimizer of \eqref{eq:min} in $\mathring{\mathcal{A}}_{h,\delta}$ with 
$\delta \le \delta_0$ for some $\delta_0>0$.  
First, suppose $u^{n+1}$ is the minimizer in $\mathring{\mathcal{A}}_{h,\delta}$ of the convex optimization problem \eqref{eq:min}, it follows that
  \begin{align}\label{jv}
    \left\langle \frac{\partial \mathcal{J}}{\partial u} (u^{n+1}), v \right\rangle = 0, \quad \forall v \in  \mathring{\C}_{\rm per},
     ~~s.t.~~  \sum_{i=1}^N v_i=0. 
  \end{align}
Hence, we must  have  
  \begin{align*}
    \frac{1}{\Delta t} \LL^{-1}_{\hat{u}^n}(u^{n+1}-u^n) + \frac{\partial F_h}{\partial u}(u^{n+1}) = {\rm const},
  \end{align*}
which yields  
  \begin{align*}
    \frac{u^{n+1}-u^n}{\Delta t} = -\LL_{\hat{u}^n}\left(\frac{\partial F_h}{\partial u}(u^{n+1}) +{\rm const}\right) = d_h\left(\hat{u}^n D_h\left(\frac{\partial F_h}{\partial u}(u^{n+1}) \right)\right),
  \end{align*}
  which is the numerical scheme \eqref{eq:timed}-\eqref{eq:mudiscrete}. On the other hand for numerical solution $u^{n+1}>0$,  we take $\delta \leq  \delta_0 < \min u^{n+1}$,  and reverse the above calculation to get (\ref{jv}), which when combined with the  convexity of $\mathcal{J}(u)$  and the assumption that the minimizer does not touch the boundary of $\mathcal{A}_{h,\delta}$, implies that it is indeed the minimizer. 
  
 Step 3. We proceed to show that there exists  $\delta_0>0$ such that the minimizer does not touch the boundary of $\mathcal{A}_{h,\delta}$  for all $0<\delta \le \delta_0$.
 We use a contradiction argument: 
 suppose there exists a minimizer $u^*$ to the optimization problem \eqref{eq:min} touching the boundary of $\mathcal{A}_{h,\delta}$ at some grid points $i_0,\ldots,i_{m-1}$ with $1\le m \le N-1$, that is
  \[u^*_{i_0} =\ldots=u_{i_{m-1}}^*= \delta.\]
 We calculate the directional derivative along $v=(v_1,\ldots,v_N) \in \mathring{\C}_{\rm per}$ with $u^*+sv \in \mathcal{A}_{h,\delta}$ as
    \begin{align}\label{eq:phical}
    \frac{d}{ds} & \mathcal{J}(u^*+sv) \bigg|_{s=0}\nonumber \\ \nonumber
    &=  \frac{d}{ds} \bigg(\frac{1}{2\Delta t}\|u^*+sv-u^n\|^2_{\LL_{\hat u^n}^{-1}}+\frac{1}{2h} \sum_{i=1}^N \frac{(u^*_{i+1}-u^*_{i}+sv_{i+1}-sv_i)^2}{u^*_i+sv_i}\bigg)\bigg|_{s=0} \nonumber\\
     &=  \bigg(\frac{1}{\Delta t} \langle \LL_{\hat{u}^n}^{-1} (u^*+sv-u^n),v\rangle  + \frac{1}{h}\sum_{i=1}^N \frac{(u^*_{i+1}-u^*_{i}+sv_{i+1}-sv_i)(v_{i+1}-v_i)}{u^*_i+sv_i} \nonumber \\
     &\quad - \frac{1}{2h}\sum_{i=1}^N\frac{(u^*_{i+1}-u^*_{i}+sv_{i+1}-sv_i)^2}{(u^*_i+sv_i)^2}v_i\bigg)\bigg|_{s=0} \nonumber\\
     &=  \frac{1}{\Delta t} \langle \LL_{\hat{u}^n}^{-1}(u^*-u^n),v\rangle + \frac{1}{h}\sum_{i=1}^N \frac{(u^*_{i+1}-u^*_{i})(v_{i+1}-v_i)}{u^*_i} \nonumber\\
     &\quad- \frac{1}{2h}\sum_{i=1}^N\frac{(u^*_{i+1}-u^*_{i})^2v_i}{(u^*_i)^2}. 
  \end{align}
  In order to find a direction along which the above derivative is negative, we distinguish two cases 
characterized by a parameter $q$ to be determined later:
  \begin{enumerate}[label=(\roman*)]
      \item \label{i1} There exists a grid index $j_0\in\{i_0,\ldots,i_{m-1}\}$ such that 
      $u^*_{j_0+1}-u^*_{j_0} > qu^*_{j_0}$. 
    \item \label{i2} For all $j\in\{i_0,\ldots,i_{m-1}\}$, $u^*_{j+1}-u^*_j \le qu^*_j$.
  \end{enumerate}
  First we consider  case \ref{i1}. Suppose $u^*$ reaches its maximum at the grid point $i_m$,
  we take the direction $v$ in the equation \eqref{eq:phical} to be
  \begin{align*}
  v_i = \left\{\begin{array}{cl} 1,& \;\text{ for } i=i_0,i_1,\ldots,i_{m-1}, \\
  -m,& \;\text{ for } i=i_m,\\
  0,&\; \text{ otherwise.}  
  \end{array}
  \right.
  \end{align*}
  It is easy to check that $u^*+sv \in \mathring{\mathcal{A}}_{h,\delta}$ for small $s$.

We illustrate by taking the case $N=3$ as an example. When $m=1$, suppose $K/h=1$ and $i_0=2$, the region of $\mathcal{A}_{h,\delta}$ is plotted in Figure \ref{fig:choice}(a). We use a ternary diagram to represent the set $\{(u_1,u_2, u_3):u_1+u_2+u_3=K/h\}$. The region inside the outer trianlge is $\mathcal{A}_{h,0}$ and the region inside the inner triangle is $\mathcal{A}_{h,\delta}$. The red point represents the minimizer $(u^*_1,u^*_2, u_3^*)$. Here $u^*_1>u^*_3>u^*_2$ so $i_1=1$. The vector $v=(-1,1,0)$ is chosen to be in the direction as shown by the red arrow in the figure.
  When $m=2$, suppose $i_0=1,i_1=2$, the vector $v=(1,1,-2)$ is chosen to be in the direction as shown by the red solid arrow in Figure \ref{fig:choice}(b). The two dashed arrows represent the vectors $(1,0,-1)$ and $(0,1,-1)$, respectively.

\begin{figure}
\begin{center}
\begin{tabular}{ccc}
    \includegraphics[width=0.33\textwidth]{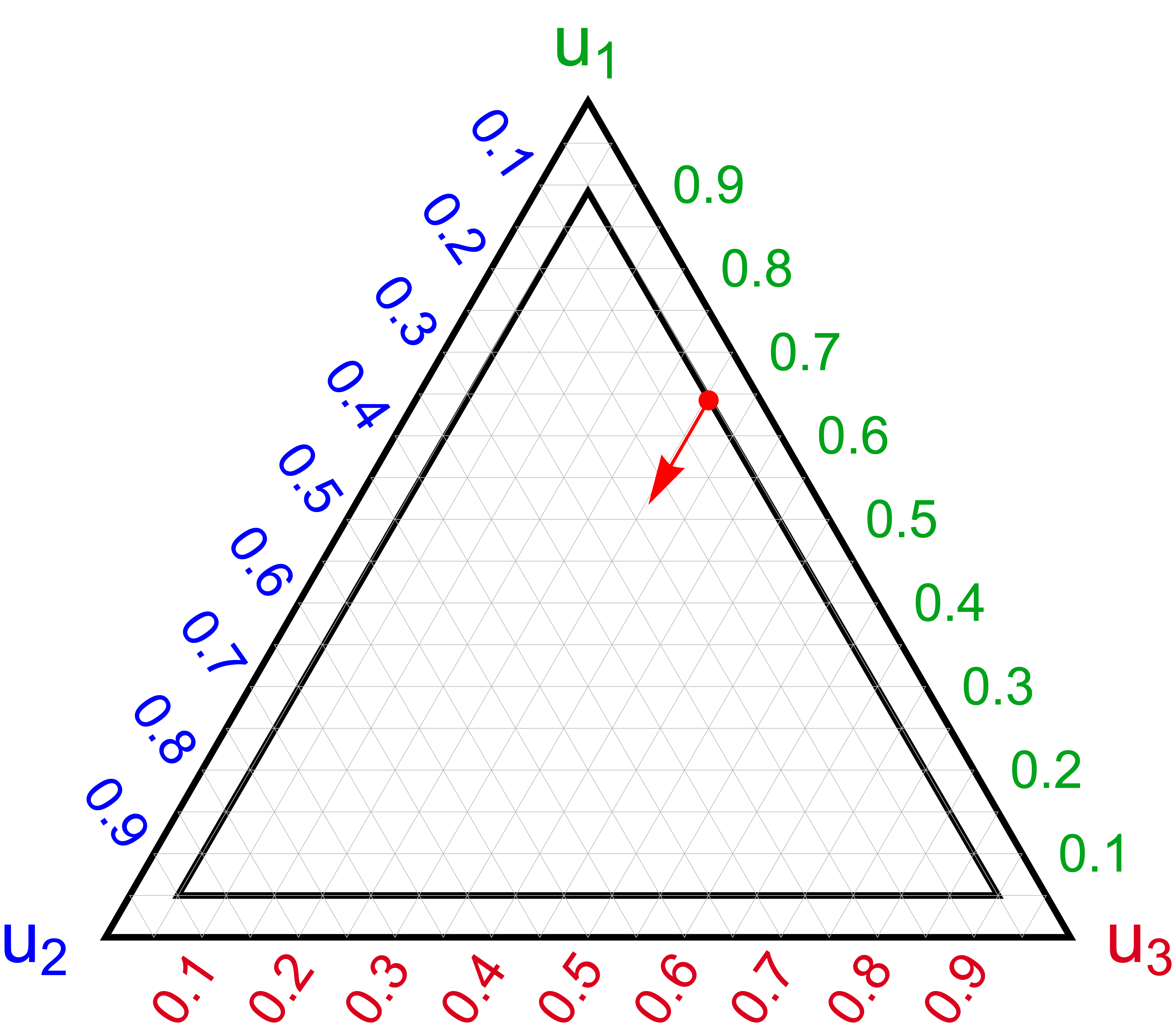} & \includegraphics[width=0.33\textwidth]{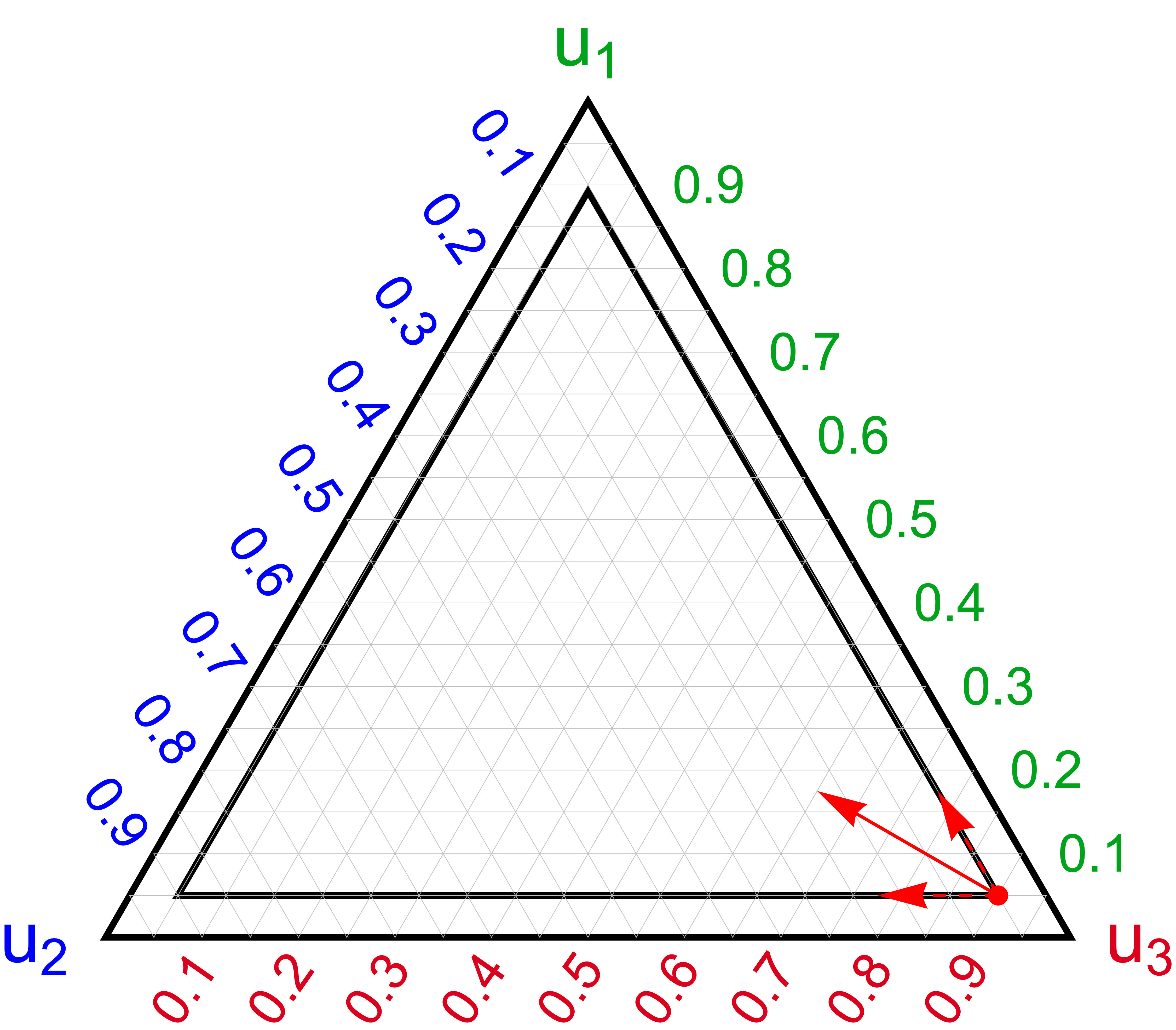} & \includegraphics[width=0.33\textwidth]{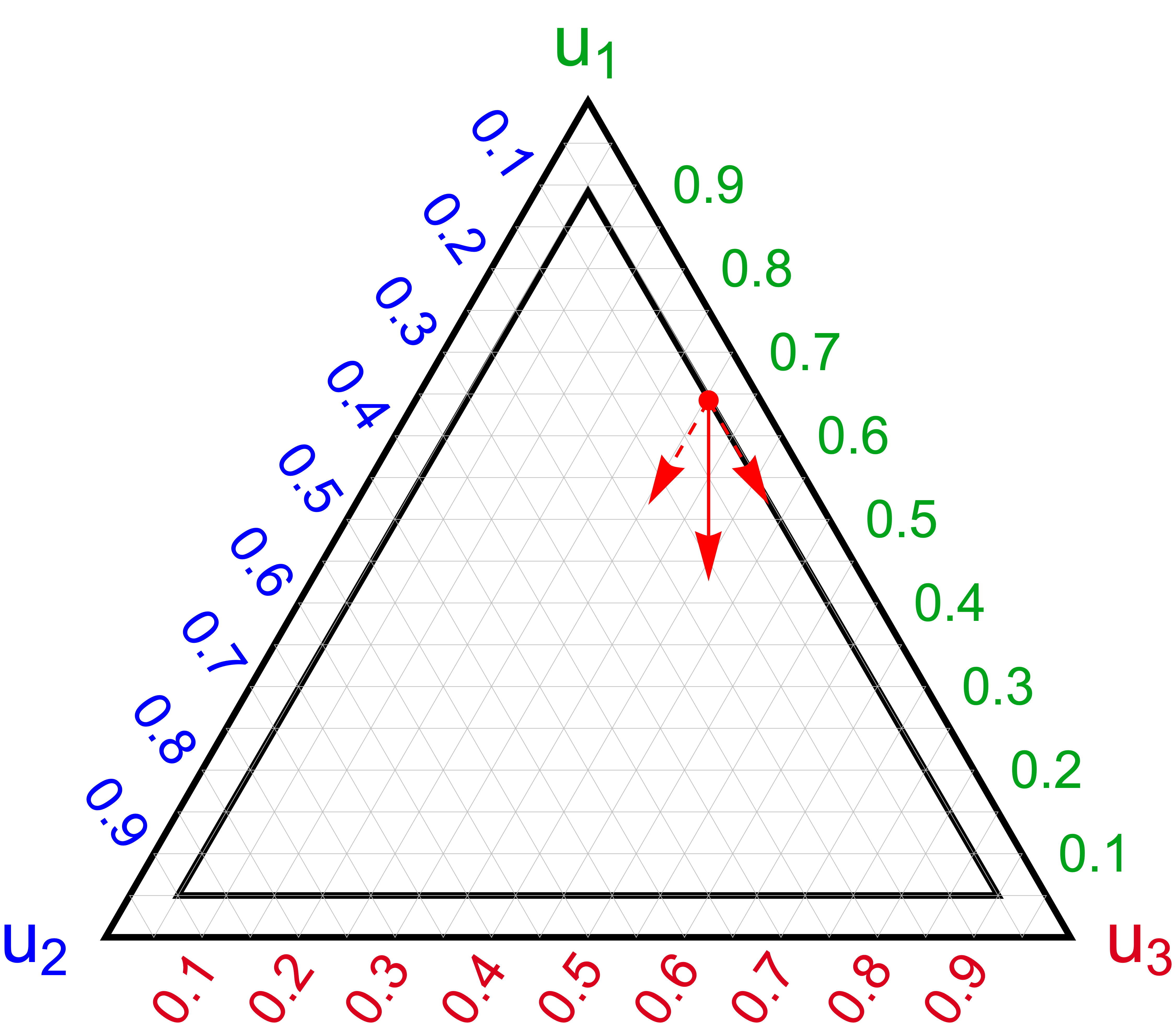} \\
    (a) & (b) & (c)
\end{tabular}
\end{center}
\caption{The direction $v$}\label{fig:choice}
\end{figure}


  The directional derivative along the vector $v$ is
  \begin{align}
     \frac{1}{h}\frac{d}{ds} & \mathcal{J}(u^*+sv) \bigg|_{s=0} \nonumber\\
    =& \frac{1}{\Delta t}\sum_{p=0}^{m-1} (\LL_{\hat{u}^n}^{-1} (u^*-u^n))_{i_p}
    - \frac{m}{\Delta t} (\LL_{\hat{u}^n}^{-1} (u^*-u^n))_{i_m} \nonumber\\
    &- \sum_{p=0}^{m-1} \left(\frac{u^{*}_{i_p+1}-u^*_{i_p}}{h^2 u^*_{i_p}} - \frac{u^{*}_{i_p}-u^*_{i_p-1}}{h^2 u^{*}_{i_p-1}}  + \frac{(u^{*}_{i_p+1}-u^*_{i_p})^2}{2h^2(u^*_{i_p})^2}\right)\nonumber \\
    &+ \frac{m(u^{*}_{i_m+1}-u^*_{i_m})}{h^2 u^*_{i_m}} -\frac{m(u^{*}_{i_m}-u^*_{i_m-1})}{h^2 u^{*}_{i_m-1}}  +\frac{m(u^{*}_{i_m+1}-u^*_{i_m})^2}{2h^2(u^*_{i_m})^2}. \label{eq:jcal1}
  \end{align}
  Since $u^*_j > \delta$ for $j\neq i_0,\ldots,i_{m-1}$, the adjacent grid points $i_p+1$ and $i_p-1$ of $i_p$ are either greater than $\delta$ or equal to $\delta$, so
  \begin{align*}
    u^*_{i_p+1}-u^*_{i_p} \ge 0, \quad u^*_{i_p}-u^*_{i_p-1} { \le} 0.
  \end{align*}
  From the assumption \ref{i1}, there exists an index $j_0 \in \{i_0,\ldots,i_{m-1}\}$ such that
  \begin{align*}
    u^*_{j_0+1}-u^*_{j_0} > qu^*_{j_0}.
  \end{align*}
  We can deduce that
  \begin{align}
    \sum_{p=0}^{m-1} \left(\frac{u^{*}_{i_p+1}-u^*_{i_p}}{h^2 u^*_{i_p}} - \frac{u^{*}_{i_p}-u^*_{i_p-1}}{h^2 u^{*}_{i_p-1}}  + \frac{(u^{*}_{i_p+1}-u^*_{i_p})^2}{2h^2(u^*_{i_p})^2}\right) \ge \frac{u^*_{j_0+1}-u^*_{j_0}}{h^2u^*_{j_0}} > \frac{q}{h^2} .\label{eq:adsf}
  \end{align}
  Since $u^*_{i_m}$ is the maximum point and the summation $\sum_{i=1}^K h u^*_i = K$, we have
  \begin{align*}
     \frac{K}{hN} < u^*_{i_m} < \frac{K}{h},
  \end{align*}
and so 
\begin{align}\label{eq:umaxin}
    \frac{(u^*_{i_m+1}-u^*_{i_m})^2}{(u^*_{i_m})^2} < \frac{2K^2}{h^2(\frac{K}{hN})^2} \le 2 N^2.
  \end{align}
  Taking the above inequality and \eqref{eq:adsf} into \eqref{eq:jcal1} and applying Lemma \ref{lm:phi}, we can deduce that
  \begin{align*}
    \frac{1}{h}\frac{d}{ds} & \mathcal{J}(u^*+sv) \bigg|_{s=0} < \frac{4mCh^{-\frac12}LK}{\Delta t \min \hat{u}^n  h} - \frac{q}{h^2} + \frac{mN^2}{h^2}.
  \end{align*}
  Taking 
  \begin{align*}
    q \ge \frac{4NCh^{\frac12}LK}{\Delta t \min \hat{u}^n}+N^3
  \end{align*}
  leads to 
  \begin{align*}
    \frac{1}{h}\frac{d}{ds} & \mathcal{J}(u^*+sv) \bigg|_{s=0} < -\frac{4(N-m)Ch^{-\frac32}LK}{\Delta t \min \hat{u}^n}-\frac{(N-m)N^2}{h^2}  <0, 
  \end{align*}
 which contradicts to the assumption that $u^*$ is a minimizer. Therefore, case \ref{i1} cannot occur.

  Next we consider the case \ref{i2}. We first claim that there exists at least one index $i_m$ such that
  \begin{align}\label{eq:ubda}
    u^*_{i_m+1} - u^*_{i_m} > qu^*_{i_m}.
  \end{align}
  Otherwise, we would have 
  $u^*_{j+1}-u^*_j \le q u_j^*$ for all $j\in \{1,\ldots,N\}$. Using this we obtain 
  \begin{align*}
    &u^*_{i_0} = \delta,\\
    &u^*_{i_0+1} = u^*_{i_0+1} -u^*_{i_0} + u^*_{i_0} \le q\delta+\delta,\\
    & u^*_{i_0+2} = u^*_{i_0+2} - u^*_{i_0+1} + u^*_{i_0+1} \le (1+q)(q\delta +\delta),\\
    & \qquad\ldots, \\
    &u^*_{i_0+N-1} = u^*_{i_0+N-1} -u^*_{i_0+N-2} + u^*_{i_0+N-2} \le (1+q)^{N-1}\delta.  
  \end{align*}
  Taking the summation of above equations gives
  \begin{align*}
    h\sum_{j=1}^N u^*_{j} \le \frac{(1+q)^N-1}{q}\delta h \le K/2
  \end{align*}
  if we take 
  \begin{align*}
   \delta \le \delta_0 = \frac{qK}{2((1+q)^N-1)h}.
  \end{align*}
  This is contradictory to the relation of $\sum_{j=1}^N u^*_j = K$.
 We can further assume 
  \begin{align}\label{eq:akdj}
      u^*_{i_m}-u^*_{i_m-1} \le qu^*_{i_m-1}.
  \end{align}
  because if this does not hold, we can move to the next grid point until such an inequality holds (notice that \eqref{eq:akdj} holds for all grid points in the set $\{i_0,\ldots,i_{m-1}\}$).

  We suppose $u^*$ reaches it maximum at the grid point $i_{m+1}$. Notice that $i_{m+1} \neq i_0,\ldots,i_{m-1}$ and $i_{m+1} \neq i_{m}$ because ${i_m}$ can not be a maximum point due to \eqref{eq:ubda}. We take the direction $v$ to be
  \begin{align*}
  v_i = \left\{\begin{array}{cl} 1,& \;\text{ for } i=i_0,i_1,\ldots,i_{m-1},i_m,\\
  -m-1,&\;\text{ for }i=i_{m+1},\\
  0,&\; \text{ otherwise.}  
  \end{array}
  \right.
  \end{align*}
  It is easy to check $u^*+sv \in \mathring{A}_{h,\delta}$ for small $s$.

   Taking $N=3$ as an example. When $m=1$, suppose $K/h=1$ and $i_0=2,i_1=3,i_2=1$, the region of $\mathcal{A}_{h,\delta}$ is plotted in Figure \ref{fig:choice}(c). 
   The vector $v=(-2,1,1)$ is chosen to be in the direction as show by the red solid arrow in the figure where the two dotted arrows represent the directions of $(-1,1,0)$ and $(-1,0,1)$. 

  The directional derivative of the objective function $\mathcal{J}(u)$ along the above direction $v$ is 
  \begin{align}
     \frac{1}{h}\frac{d}{ds} & \mathcal{J}(u^*+sv) \bigg|_{s=0} \nonumber\\
    =& \frac{1}{\Delta t}\sum_{p=0}^m (\LL_{\hat{u}^n}^{-1} (u^*-u^n))_{i_p}
    - \frac{m+1}{\Delta t} (\LL_{\hat{u}^n}^{-1} (u^*-u^n))_{i_m} \nonumber\\
    &- \sum_{p=0}^{m-1} \left(\frac{u^{*}_{i_p+1}-u^*_{i_p}}{h^2 u^*_{i_p}} - \frac{u^{*}_{i_p}-u^*_{i_p-1}}{h^2 u^{*}_{i_p-1}}  + \frac{(u^{*}_{i_p+1}-u^*_{i_p})^2}{2h^2(u^*_{i_p})^2}\right)\nonumber \\
    &-\frac{u^{*}_{i_m+1}-u^*_{i_m}}{h^2 u^*_{i_m}} + \frac{u^{*}_{i_m}-u^*_{i_m-1}}{h^2 u^{*}_{i_m-1}}  - \frac{(u^{*}_{i_m+1}-u^*_{i_m})^2}{2h^2(u^*_{i_m})^2} 
    + \frac{(m+1)(u^{*}_{i_{m+1}+1}-u^*_{i_{m+1}})}{h^2 u^*_{i_{m+1}}}\nonumber \\&-\frac{(m+1)(u^{*}_{i_{m+1}}-u^*_{i_{m+1}-1})}{h^2 u^{*}_{i_{m+1}-1}}  +\frac{(m+1)(u^{*}_{i_{m+1}+1}-u^*_{i_{m+1}})^2}{2h^2(u^*_{i_{m+1}})^2}. \label{eq:jcal2}
  \end{align}
  From the inequalities \eqref{eq:ubda} and \eqref{eq:akdj}, we have 
  \begin{align*}
    -\frac{u^{*}_{i_m+1}-u^*_{i_m}}{h^2 u^*_{i_m}} + \frac{u^{*}_{i_m}-u^*_{i_m-1}}{h^2 u^{*}_{i_m-1}}  - \frac{(u^{*}_{i_m+1}-u^*_{i_m})^2}{2h^2(u^*_{i_m})^2} < - \frac{q}{h^2} + \frac{q}{h^2} - \frac{q^2}{2h^2} = -\frac{q^2}{2h^2}.
  \end{align*}
  Since $u^*_j > \delta$ for all $j\in\{1,\ldots,N\}$ and $j\not\in\{i_0,\ldots,i_{m-1}\}$ and $u^*_{i_{m+1}}$ is the maximum point so \eqref{eq:umaxin} holds, we have
  \begin{align*}
    \frac{1}{h}\frac{d}{ds} & \mathcal{J}(u^*+sv) \bigg|_{s=0} \le \frac{4(m+1)Ch^{-\frac12}LK}{\Delta t \min \hat{u}^n h} -\frac{q^2}{2h^2}+\frac{(m+1)N^2}{h^2}.
  \end{align*}
  Taking 
    \[q\ge\left(\frac{8(N+1)Ch^{\frac12}LK}{\Delta t\min\hat{u}^n}\right)^{\frac12} +\sqrt{2(N+1)}N ,\]
    we get
    \begin{align*}
          \frac{1}{h}\frac{d}{ds} & \mathcal{J}(u^*+sv) \bigg|_{s=0} \le -\frac{4(N-m)Ch^{-\frac32}LK}{\Delta t \min \hat{u}^n} - \frac{(N-m)N^2}{h^2}<0,
    \end{align*}
    which contradicts to the assumption that $u^*$ is a minimizer. Therefore, the situation \ref{i2} cannot occur, either.

  Combining the above estimates, we set 
    \[q=\max\left\{\frac{4NCh^{\frac12}LK}{\Delta t \min \hat{u}^n}+N^3, \left(\frac{8(N+1)Ch^{\frac12}LK}{\Delta t\min\hat{u}^n}\right)^{\frac12} +\sqrt{2(N+1)}N\right\}\]
    and    
    $$ \delta_0 =\frac{qK}{2((1+q)^N-1)h}$$
    so that for $\delta \le \delta_0$  neither \ref{i1} nor \ref{i2} can occur.  Hence no minimizer $u^*$ of the optimization problem \eqref{eq:min} can touch the boundary of $\mathcal{A}_{h,\delta}$  for any $\delta \le \delta_0$. 
Thus we finish the proof.
  \end{proof}

The mass conservation, positivity-preserving,  and energy stability of the scheme \eqref{eq:timed}-\eqref{eq:mudiscrete} can be obtained directly from Theorem \ref{thm}. 
We state the results in the following theorem.
\begin{thm}\label{thm:property}
  Given $u^0 \in \mathcal{C}_{\rm per}$ with $u^0>0$. The numerical scheme \eqref{eq:timed}-\eqref{eq:mudiscrete} has a unique solution $u^n$ for $n\ge 1, n\in\mathbb{N}$ satisfying the following properties:
  \begin{enumerate}
    \item (Mass conservation)  For any $n\ge 1, n\in\mathbb{N}$, \[\sum_{i=1}^N h u_i^{n} = \sum_{i=1}^N h u_{i}^{0}.\]
    \item (Positivity-preserving) For any $i=1,\ldots,N$, 
    \[u_i^{n}>0\]
    holds for any $n\ge 1, n\in\mathbb{N}$.
    \item  (Unconditionally energy stability) For any $n\ge 0, n\in\mathbb{N}$,
  \begin{align}\label{eq:Fhmon}
    F_h(u^{n+1}) \le F_h(u^n) - \frac{1}{\Delta t}\|u^{n+1}-u^n\|_{\LL_{\widehat u^n}^{-1}}^2
  \end{align}
  holds for any $\Delta t>0$.
  \end{enumerate}
\end{thm}
\begin{proof}
(1) This property is ensured by the conservative form of the numerical scheme, leading to 
\[\sum_{i=1}^N u_i^{n+1} = \sum_{i=1}^N u^n_i.
\]

(b) Starting from $u^0>0$, we apply Theorem \ref{thm} recursively to obtain 
$$
u^{n} \in  \mathring{\mathcal{A}}_{h,\delta_n}\subset \mathring{\mathcal{A}}_{h,0}
$$ 
for some constant $\delta_n>0$, together we have  for any $n$,  
\[
u^n \in \bigcap_{i=1}^{\infty} \mathring{\mathcal{A}}_{h,\delta_i}\subset\mathring{\mathcal{A}}_{h, 0}.
\]
This ensures that $u^n>0$. 

(3) Since the solution of the numerical scheme \eqref{eq:timed}-\eqref{eq:mudiscrete} is the minimizer of the optimization problem \eqref{eq:min}, we have
  \begin{align*}
    \mathcal{J}[u^{n+1}] \le \mathcal{J}[u^n],
  \end{align*}
  that is
  \begin{align*}
    \frac{1}{\Delta t}\|u^{n+1}-u^n\|_{\LL_{\widehat u^n}^{-1}}^2 + F_h(u^{n+1}) \le F_h(u^n).
  \end{align*}
  Therefore inequality \eqref{eq:Fhmon} holds.
\end{proof}

\begin{rem}
 In the proof of Theorem \ref{thm}, $\delta$ is dependent of $h, \Delta t$ and $\hat{u}^n$.  
Moreover, $q$ is not guaranteed to be finite when $h\to 0$ and $\Delta t \to 0$, the above analysis cannot be carried over to the continuous equation. 
\end{rem}

\begin{rem}  A standard central discretization of equation \eqref{eq:DLSSnd} is given by
  \[H = - \frac{|D_h u|^2}{2\hat{u}^2} - d_h\left(\frac{D_h u}{ \hat{u}}\right),\]
  which corresponds to the derivative of a discretized energy
  \[
  F_h(u):= \frac{1}{2}\left[ \frac{D_h u}{\hat{u}}, D_h u \right] = \frac{h}{2} \sum_{i=1}^N \frac{|(D_h u)_{i+\frac{1}{2}}|^2}{ \hat{u}_{i+\frac{1}{2}}} =\frac{1}{h}\sum_{i=1}^N \frac{(u_{i+1}-u_{i})^2}{(u_{i+1}+u_i)}.
  \]
  With this discrete energy Theorem \ref{thm} and Theorem \ref{thm:property} can be proven in similar manner.   Moreover,  Theorem \ref{thm} and Theorem \ref{thm:property} also hold for 
 schemes with other forms of the discrete energy such as  
  \begin{align*}
   F_h(u) = \frac{1}{h} \sum_{i=1}^N \frac{(u_i-u_{i-1})^2}{2u_i},
 \end{align*} 
 or 
 \begin{align*}
 F_h(u) =\frac{1}{4h} \sum_{i=1}^N \left(\frac{(u_{i+1}-u_i)^2 + (u_i-u_{i-1})^2}{u_i}\right). 
\end{align*}
\end{rem}

\section{The scheme in higher dimensions} 
The numerical scheme \eqref{eq:timed}-\eqref{eq:mudiscrete} based on the gradient flow structure of the equation \eqref{eq:DLSSnd} can be generalized to higher dimensions. Theorem \ref{thm} and Theorem \ref{thm:property} can be shown to hold too. We will illustrate this by only considering  the two dimensional case. 


\subsection{Notations.}
\textbf{}We take the domain to be $\mathbb{T}^2 = [0,L]\times [0,L]$, and the mesh size $h=L/N$ for both $x$ and $y$ directions. We label the grid points with $i,j$ for $i=1,\ldots, N$ and $j=1,\ldots,N$. Since we are taking a periodic grid, we extend the label to the whole $\mathbb{Z} \times \mathbb{Z}$ by using $\mathcal{C} \times \mathcal{C}$ with $\mathcal{C}$ defined in \eqref{eq:cedef} and also label the mid-point grids using $\mathcal{E} \times \mathcal{E}$ with $\mathcal{E}$ defined in \eqref{eq:cedef}. 
We define the periodic function spaces:
\begin{align*}
  \mathcal{C}_{\rm per} := \{f: \mathcal{C} \times \mathcal{C} \to \mathbb{R}\},\quad &\mathcal{E}^X_{\rm per}:=\{f: \mathcal{E} \times \mathcal{C} \to \mathbb{R}\}, \quad \mathcal{E}^Y_{\rm per}:=\{f: \mathcal{C} \times \mathcal{R} \to \mathbb{R}\}, 
\end{align*}
and also 
\begin{align*}
  \mathring{\mathcal{C}}_{\rm per} := \left\{f\in \mathcal{C}_{\rm per}: \sum_{i,j=1}^N f_{i,j}=0\right\}.
\end{align*}
We define the difference operators $D_x,D_y$ and $d_x,d_y$ as
\begin{align*}
  (D_xf)_{i+\frac{1}{2},j} = \frac{f_{i+1,j}-f_{i,j}}{h}, \quad (D_yf)_{i,j+\frac{1}{2}} = \frac{f_{i,j+1}-f_{i,j}}{h}, \\
  (d_xf)_{i,j} = \frac{f_{i+\frac{1}{2},j}-f_{i-\frac{1}{2},j}}{h}, \quad (d_yf)_{i,j} = \frac{f_{i,j+\frac{1}{2}}-f_{i,j-\frac{1}{2}}}{h}, 
\end{align*}
the discrete gradient $D_h$ as
\begin{align}\label{eq:Dh2d}
  (D_h f)_{i,j} = (D_x f_{i+\frac{1}{2},j},D_y f_{i,j+\frac{1}{2}}),
\end{align}
and the discrete divergence $d_h$ as
\begin{align}\label{eq:dh2d}
 d_h \cdot \mathbf{f}_{i,j} = d_x f^x_{i,j} + d_y f^y_{i,j},\quad \mathbf{f}_{i,j} = (f^x,f^y) \in \mathcal{E}_{\rm per}^X \times \mathcal{E}_{\rm per}^Y .
\end{align}
We also define the averages
\begin{align}\label{eq:Axax2d}
  \hat{f}_{i+\frac{1}{2},j} = \frac{f_{i,j}+f_{i+1,j}}{2},
  \quad \hat{f}_{i,j+\frac{1}{2}} = \frac{f_{i,j}+f_{i,j+1}}{2}.
\end{align}
The inner products are defined by
\begin{align*}
  &\langle f,g \rangle:=h^2 \sum_{i,j=1}^N f_{i,j}g_{i,j}, \; \forall f,g \in \mathcal{C}_{\rm per}, \\
  & [f,g]_x:= \frac{1}{2}h^2\sum_{i,j=1}^N \left(f_{i+\frac{1}{2},j} g_{i+\frac{1}{2},j}+f_{i-\frac{1}{2},j} g_{i-\frac{1}{2},j}\right), \\
  &[f,g]_y:=\frac{1}{2}h^2\sum_{i,j=1}^N \left(f_{i,j+\frac{1}{2}} g_{i,j+\frac{1}{2}}+f_{i,j-\frac{1}{2}} g_{i,j-\frac{1}{2}}\right),
\end{align*}
and for vector functions:
\begin{align*}
  [\mathbf{f},\mathbf{g}]:=[f^x,g^x]_x + [f^y,g^y]_y.
\end{align*}
The corresponding norms in these spaces are defined accordingly.
Suppose $f,g \in \mathcal{C}_{\rm per}$, and $\phi$ is defined on all the edge-center points $\mathcal{C}\times \mathcal{E} \, \cup \, \mathcal{E} \times \mathcal{C}$, then the following summation-by-parts formulas hold:
\begin{align*}
  \langle f,d_h \cdot \mathbf{g} \rangle = -[D_h f,\mathbf{g}],\quad \langle f,d_h \cdot (\phi D_h g) \rangle = - [D_h f, \phi D_h g].
\end{align*}
Suppoe $\phi \in \mathcal{C}\times \mathcal{E} \, \cup \, \mathcal{E} \times \mathcal{C}$ and $\phi>0$, we introduce the following operator $\mathcal{L}$ on $\mathring{\C}_{\rm per}$:
\begin{align}
\mathcal{L}_{\phi} f = -d_h \cdot (\phi D_h f) = g, \text{ for }f\in \mathring{\C}_{\rm per}.
\end{align}
it follows that $g \in \mathring{\mathcal{C}}_{\rm per}$. Since $\mathcal{\LL}_{\phi}$ is invertible, we can then define the norm for any $ g\in\mathring{\C}_{\rm per}$,
\begin{align}\label{eq:phinorm2d}
  \|g\|_{\mathcal{L}_\phi^{-1}}^2:=[\phi D_h f,D_h f],
\end{align}
It can also be induced by the bilinear form
\begin{align*}
  \langle g_1,g_2 \rangle_{\mathcal{L}_\phi^{-1}}:=\langle g_1,\mathcal{L}_\phi^{-1} g_2\rangle=\langle \mathcal{L}_\phi^{-1}g_1, g_2\rangle = [\phi D_h f_1,D_h f_2].
\end{align*}

\subsection{The scheme}
The scheme in abstract form is 
\begin{align}\label{eq:timed2}
  \frac{u^{n+1}-u^n}{\Delta t} = d_h \cdot (\hat{u}^n D_h H^{n+1} ), 
\end{align}
where $H^{n+1}$ is given by
\begin{align}\label{eq:mu2}
  H^{n+1} =& \frac{1}{h^2}\frac{\partial F_h(u)}{\partial u}(u^{n+1}),
\end{align}
with 
\begin{align}\label{eq:fh2d}
  F_h(u) = \frac{1}{2} \sum_{i,j=1}^N \left(\frac{(u_{i+1,j}-u_{i,j})^2}{u_{i,j}} + \frac{(u_{i,j+1}-u_{i,j})^2}{u_{i,j}} \right).
\end{align}
Its explicit form is
\begin{align*}
  \frac{u_{i,j}^{n+1} - u_{i,j}^n}{\Delta t} = \frac{1}{2h^2} \left( (u^n_{i+1,j}+u^n_{i,j})(H_{i+1,j}^{n+1}-H_{i,j}^{n+1})-(u^n_{i,j}+u^n_{i-1,j})(H_{i,j}^{n+1}-H_{i-1,j}^{n+1})\right) \\
  +\frac{1}{2h^2} \left( (u^n_{i,j+1}+u^n_{i,j})(H_{i,j+1}^{n+1}-H_{i,j}^{n+1})-(u^n_{i,j}+u^n_{i,j-1})(H_{i,j}^{n+1}-H_{i,j-1}^{n+1})\right),
\end{align*}
with
\begin{align}
  H_{i,j}^{n+1} 
   =& -\frac{1}{2h^2}\left(\frac{(u_{i+1,j}^{n+1}-u_{i,j}^{n+1})^2}{(u_{i,j}^{n+1})^2} + \frac{(u_{i,j+1}^{n+1}-u_{i,j}^{n+1})^2}{(u_{i,j}^{n+1})^2} \right) \nonumber \\
  &- \frac{1}{h^2} \left(\frac{u_{i+1,j}^{n+1}-u_{i,j}^{n+1}}{u_{i,j}^{n+1}} - \frac{u_{i,j}^{n+1}-u_{i-1,j}^{n+1}}{u_{i-1,j}^{n+1}}+ \frac{u_{i,j+1}^{n+1}-u_{i,j}^{n+1}}{u_{i,j}^{n+1}}-\frac{u_{i,j}^{n+1}-u_{i,j-1}^{n+1}}{u_{i,j-1}^{n+1}}\right).
\end{align}
Again, the numerical scheme \eqref{eq:timed2}-\eqref{eq:mu2} is mass conservative, positivity-preserving and unconditionally energy stable. We state the main result in the following.
\begin{thm}\label{thm:property+}
  Given $u^0 \in \mathcal{C}_{\rm per}$ with $u^0>0$. The numerical scheme  \eqref{eq:timed2}-\eqref{eq:mu2} has a unique solution $u^n$ for $n\ge 1, n\in\mathbb{N}$ satisfying the following properties:
  \begin{enumerate}
    \item (Mass conservation)  For any $n\ge 1, n\in\mathbb{N}$, \[\sum_{i, j=1}^N h^2 u_{i, j}^{n} 
    = \sum_{i, j=1}^N h^2 u_{i, j}^{0}.\]
    \item (Positivity-preserving) For any $i, j=1,\ldots,N$, 
    \[u_{i, j}^{n}>0\]
    holds for any $n\ge 1, n\in\mathbb{N}$.
    \item  (Unconditionally energy stability) For any $n\ge 0, n\in\mathbb{N}$,
  \begin{align}\label{eq:Fhmon+}
    F_h(u^{n+1}) \le F_h(u^n) - \frac{1}{\Delta t}\|u^{n+1}-u^n\|_{\LL_{\widehat u^n}^{-1}}^2
  \end{align}
  holds for any $\Delta t>0$.
  \end{enumerate}
\end{thm}

\subsection{Proof of Theorem \ref{thm:property+}}
We prove Theorem \ref{thm:property+} for the two dimensional case in the same way as for the one dimensional case, i.e.,  by first proving the result of  Theorem \ref{thm} for 2D case.

Since the scheme \eqref{eq:timed2}-\eqref{eq:mu2} follows the same gradient flow structure,
 we can prove that $u^{n+1}>0$ is a solution to the numerical scheme \eqref{eq:timed2}-\eqref{eq:mu2} if and only if it is an interior minimizer of the optimization problem
\begin{align}\label{eq:min2d}
u^{n+1}=\argmin_{u\in \mathcal{A}_h} \left\{ \J[u] = \frac{1}{2\Delta t}\|u-u_n\|^2_{\LL_{u^n}^{-1}} + F_h(u) \right\}
\end{align}
over the set
\begin{align}\label{eq:Aset2d}
    \mathcal{A}_{h,\delta} =\left\{ u \in \C_{\rm per} \;: \; u \ge\delta, \quad h^2\sum_{i,j=1}^N u_{i,j} = h^2\sum_{i,j=1}^N u_{i,j}^0=:K\right\}
\end{align}
with $\delta \leq \delta_0$ for some $\delta_0>0$. Here we only need to use $\|\cdot\|_{\LL^{-1}_{\hat{u}^n}}$ and $F_h$ defined in \eqref{eq:phinorm2d} and \eqref{eq:fh2d} for the two dimensional case to replace the corresponding one dimensional definitions in  Step 1 and 2 of the previous proof of Theorem \ref{thm}.  

What is left  is to prove that there exists a constant $\delta_0>0$ such that the minimizer of \eqref{eq:min2d} does not touch the boundary.
 We follow the previous proof for the one dimensional case and calculate the directional derivative at the minimizer $u^*$ which touches the boundary at $m$ grid points ($1\le m \le N^2-1$)
$$u^*_{i_0,j_0} =\ldots=u^*_{i_{m-1},j_{m-1}} = \delta,$$
 along a direction $v=(v_{i,j})_{N\times N} \in \mathring{\C}_{\rm per}$ with $u^*+sv \in \mathcal{A}_{h,\delta}$ by
    \begin{align}
    \frac{d}{ds} & \mathcal{J}(u^*+sv) \bigg|_{s=0} \nonumber\\
    =&  \frac{d}{ds} \bigg(\frac{1}{2\Delta t}\|u^*+sv-u^n\|^2_{\LL_{\widehat u^n}^{-1}} \nonumber\\
     &+\frac{1}{2} \sum_{i=1}^N \frac{(u^*_{i+1,j}-u^*_{i,j}+sv_{i+1,j}-sv_{i,j})^2+(u^*_{i,j+1}-u^*_{i,j}+sv_{i,j+1}-sv_{i,j})^2}{u^*_{i,j}+sv_{i,j}}\bigg)\bigg|_{s=0} \nonumber\\
     =&  \bigg(\frac{1}{\Delta t} \langle \LL_{\widehat{u}^n}^{-1} (u^*-u^n),v\rangle  + \sum_{i=1}^N \frac{(u^*_{i+1,j}-u^*_{i,j})(v_{i+1,j}-v_{i,j})}{u^*_{i,j}} \nonumber \\
     &  + \sum_{i=1}^N \frac{(u^*_{i,j+1}-u^*_{i,j})(v_{i,j+1}-v_{i,j})}{u^*_{i,j}}  - \sum_{i=1}^N\frac{(u^*_{i+1,j}-u^*_{i,j})^2+(u^*_{i,j+1}-u^*_{i,j})^2}{2(u^*_{i,j})^2}v_{i,j}\bigg). \label{eq:jd2d}
  \end{align}
Likewise we distinguish the cases based on a positive constant $q$ to be determined:
  \begin{enumerate}[label=(\roman*)]
      \item \label{i12} There exists an index $ r\in\{0,\ldots,m-1\}$ such that $u^*$  satisfies $u^*_{i_r+1,j_r}-u^*_{i_r,j_r} > qu^*_{i_r,j_r}$ or $u^*_{i_r,j_r+1}-u^*_{i_r,j_r} > qu^*_{i_r,j_r}$ for some $ r\in\{0,\ldots,m-1\}$.
    \item \label{i22} For all 
    $(i,j)=(i_0,j_0),\ldots,(i_{m-1},j_{m-1})$
    , the inequalities $u^*_{i+1,j}-u^*_{i,j} \le qu^*_{i,j}$ and $u^*_{i,j+1}-u^*_{i,j} \le qu^*_{i,j}$ hold.
  \end{enumerate}
First we consider the case \ref{i12}. Suppose $u^*$ reaches its maximum at $i=i_m,j=j_m$ with value $u^*_{i_{m},j_{m}}$. We take
  \[v_{i,j} = \left\{\begin{array}{cl} 1,&\text{ for } (i,j)=(i_0,j_0),\ldots,(i_{m-1},j_{m-1}), \\
  -m,&\text{ for } (i,j) = (i_m,j_m),\\
  0,&\text{ otherwise,}  \end{array}\right.\]
  and the equation \eqref{eq:jd2d} becomes
  \begin{align}\label{eq:dsphi2d}
   \frac{1}{h^2}\frac{d}{ds} & \mathcal{J}(u^*+sv) \bigg|_{s=0} \nonumber\\
    =& \sum_{p=0}^{m-1}\frac{1}{\Delta t} (\LL_{\hat{u}^n}^{-1} (u^*-u^n))_{i_p,j_p} - \frac{m}{\Delta t} (\LL_{\hat{u}^n}^{-1} (u^*-u^n))_{i_m,j_m} \nonumber\\ 
    &- \sum_{p=0}^{m-1}\bigg(\frac{1}{h^2} \frac{u^{*}_{i_p+1,j_p}-u^*_{i_p,j_p}}{u^*_{i_p, j_p}}
    - \frac{1}{h^2} \frac{u^{*}_{i_p,j_p}-u^*_{i_p-1,j_p}}{u^{*}_{i_p-1,j_p}}  + \frac{1}{h^2}\frac{u^{*}_{i_p,j_p+1}-u^*_{i_p,j_p}}{u^*_{i_p, j_p}} \nonumber \\
    &- \frac{1}{h^2}\frac{u^{*}_{i_p,j_p} -u^*_{i_p,j_p-1}}{u^{*}_{i_p,j_p-1}} +\frac{1}{h^2} \frac{(u^{*}_{i_p+1,j_p}-u^*_{i_p,j_p})^2 + (u^{*}_{i_p,j_p+1}-u^*_{i_p,j_p})^2}{2(u^*_{i_p,j_p})^2} \bigg) \nonumber \\
    &+ \frac{m}{h^2} \frac{u^{*}_{i_m+1,j_m}-u^*_{i_m,j_m}}{u^*_{i_m, j_m}} -\frac{m}{h^2} \frac{u^{*}_{i_m,j_m}-u^*_{i_m-1,j_m}}{u^{*}_{i_m-1,j_m}}  + \frac{m}{h^2}\frac{u^{*}_{i_m,j_m+1}-u^*_{i_m,j_m}}{u^*_{i_m, j_m}} \nonumber \\
    &- \frac{m}{h^2}\frac{u^{*}_{i_m,j_m} -u^*_{i_m,j_m-1}}{u^{*}_{i_m,j_m-1}} +\frac{m}{h^2} \frac{(u^{*}_{i_m+1,j_m}-u^*_{i_m,j_m})^2 + (u^{*}_{i_m,j_m+1}-u^*_{i_m,j_m})^2}{2(u^*_{i_m,j_m})^2}. 
  \end{align}  
 Since $\{u^*_{i_p,j_p}\}_{p=0}^m$ are the minimum points and $u^*_{i_m,j_m}$ is the maximum point, we have
 \begin{align}
   &u^*_{i_p,j_p} \le u^*_{i_p+1,j_p},u^*_{i_p,j_p+1},u^*_{i_p-1,j_p},u^*_{i_p,j_p-1},\quad p=0,\ldots,m-1, \label{eq:adf1}\\
   &u^*_{i_m,j_m} \ge u^*_{i_m+1,j_m},u^*_{i_m,j_m+1},u^*_{i_m-1,j_m},u^*_{i_m,j_m-1}.  \label{eq:adf2} 
 \end{align}
 Since 
 \begin{align*}
   h^2\sum_{i,j=1}^N u^*_{i,j}=K,
 \end{align*}
 it holds that
 \begin{align*}
   \frac{K}{h^2N^2} \le u^*_{i_m,j_m} \le \frac{K}{h^2}.
 \end{align*}
 Therefore,
 \begin{align} \label{eq:adf3}
   \frac{(u^{*}_{i_m+1,j_m}-u^*_{i_m,j_m})^2 + (u^{*}_{i_m,j_m+1}-u^*_{i_m,j_m})^2}{2(u^*_{i_m,j_m})^2}\le \frac{4K^2}{2h^4\frac{K^2}{h^4N^4}} =2N^4.
 \end{align}
 Taking the inequalities \eqref{eq:adf1}-\eqref{eq:adf3} into \eqref{eq:dsphi2d} and using Lemma \ref{lm:phi}, which also holds for the two dimensions \cite{dong2018positivity},
  we get
  \begin{align}\label{eq:ds1}
      \frac{1}{h^2}  \frac{d}{ds}& \mathcal{J}(u^*+sv) \bigg|_{s=0} \nonumber\\
    \le& \sum_{p=0}^{m-1} \frac{1}{\Delta t} (\LL_{\widehat{u}^n}^{-1} (u^*-u^n))_{i_p,j_p} - \frac{m}{\Delta t} (\LL_{\widehat{u}^n}^{-1} (u^*-u^n))_{i_m,j_m} \nonumber\\
    &- \frac{1}{h^2}\frac{u^{*}_{i_r+1,j_r}-u^*_{i_r,j_r}}{u^*_{i_r, j_r}} - \frac{1}{h^2}\frac{u^{*}_{i_r,j_r+1}-u^*_{i_r,j_r}}{u^*_{i_r, j_r}} \nonumber\\
    &+\frac{m}{h^2} \frac{(u^{*}_{i_m+1,j_m}-u^*_{i_m,j_m})^2 + (u^{*}_{i_m,j_m+1}-u^*_{i_m,j_m})^2}{2(u^*_{i_m,j_m})^2}  \nonumber\\
    \le& \frac{4mCh^{-\frac{1}{2}} L K }{\Delta t\min{\widehat{u}^n}h^2} - \frac{q}{h^2} + \frac{2mN^4}{h^2}.
  \end{align}
  We take
  \[q\ge\frac{4N^2Ch^{-\frac12}LK}{\Delta t \min \hat{u}^n}+2N^6,\]
  so that 
  $$\frac{d}{ds}\mathcal{J}(u^*+sv)\bigg|_{s=0}\le- \frac{4(N^2-m)Ch^{-\frac52}LK}{\Delta t \min \hat{u}^n} - \frac{2(N^2-m)N^4}{h^2}<0.
  $$
This says that $u^*$ cannot be a minimizer, which is contradictory to our assumption.

Next we consider the case \ref{i22}. Similar to \eqref{eq:ubda}, we can find at least one index $i_{m},j_m$ such that 
  \begin{align}\label{eq:bdf1}
    u^*_{i_m+1,j_m} - u^*_{i_m,j_m} > q u^*_{i_m,j_m} \text{ or } u^*_{i_m,j_m+1} - u^*_{i_m,j_m} > qu^*_{i_m,j_m}
  \end{align}
  holds. Otherwise, the inequalities
  \begin{align*}
   u^*_{i+1,j}-u^*_{i,j} \le qu^*_{i,j},~~u^*_{i,j+1}-u^*_{i,j} \le qu^*_{i,j}
  \end{align*}
  hold for all $i,j\in \{1,\ldots,N\}$, then we have
  $u^*_{i_0+p,j_0} \le (1+q)^pu^*_{i_0,j_0}=(1+q)^p\delta$ and $u^*_{i_0,j_0+m} \le (1+q)^mu^*_{i_0,j_0}=(1+q)^m\delta$, and
  \[u^*_{i_0+p,j_0+m} \le (1+q)^pu^*_{i_0,j_0+m}\le (1+q)^{pm} u^*_{i_0,j_0}=(1+q)^{pm} \delta.\]
  Taking the summation over $p=0,\ldots,N-1$ and $m=0,\ldots,N-1$ leads to 
  \begin{align*}
   h^2 \sum_{i,j=1}^N u^*_{i,j} = \sum_{p,m=0}^{N-1} (1+q)^{pm} \delta h^2.
  \end{align*}
  Taking 
  \begin{align*}
    \delta \le \delta_0= \frac{K}{2h^2\sum_{p,m=0}^{N-1}(1+q)^{pm}}
  \end{align*}
   leads to a contradiction as 
   \begin{align*}
     K=h^2\sum_{i,j=1}^N u^*_{i,j} \le \frac{K}{2}.
   \end{align*}
   Therefore, there exists at least one index $i_m,j_m$ satisfying \eqref{eq:bdf1}. 
  We can also assume that
  \begin{align}\label{eq:bdf2}
    u^*_{i_m,j_m}-u^*_{i_m-1,j_m} \le qu^*_{i_m-1,j_m} \text{ and } u^*_{i_m,j_m}-u^*_{i_m,j_m-1} \le qu^*_{i_m,j_m}.
  \end{align}
  Suppose $u^*$ reaches its maximum at $i_{m+1},j_{m+1}$,
  we take the direction $v$ to be 
    \[v_{i,j} = \left\{\begin{array}{cl} 1, & \;\text{ for } (i,j)=(i_0,j_0),\ldots,(i_m,j_m), \\
  -(m+1),&\;\text{ for } (i,j) =(i_{m+1},j_{m+1}), \\
  0,&\; \text{ otherwise.}  
  \end{array}
  \right.
  \]
  The directional derivative \eqref{eq:jd2d} then becomes
  \begin{align}
   \frac{1}{h^2}\frac{d}{ds} & \mathcal{J}(u^*+sv) \bigg|_{s=0} \nonumber\\
    =& \sum_{p=0}^{m}\frac{1}{\Delta t} (\LL_{\hat{u}^n}^{-1} (u^*-u^n))_{i_p,j_p} - \frac{m+1}{\Delta t} (\LL_{\hat{u}^n}^{-1} (u^*-u^n))_{i_m,j_m} \nonumber\\ 
    &- \sum_{p=0}^{m}\bigg(\frac{1}{h^2} \frac{u^{*}_{i_p+1,j_p}-u^*_{i_p,j_p}}{u^*_{i_p, j_p}}
    - \frac{1}{h^2} \frac{u^{*}_{i_p,j_p}-u^*_{i_p-1,j_p}}{u^{*}_{i_p-1,j_p}}  + \frac{1}{h^2}\frac{u^{*}_{i_p,j_p+1}-u^*_{i_p,j_p}}{u^*_{i_p, j_p}} \nonumber \\
    &- \frac{1}{h^2}\frac{u^{*}_{i_p,j_p} -u^*_{i_p,j_p-1}}{u^{*}_{i_p,j_p-1}} +\frac{1}{h^2} \frac{(u^{*}_{i_p+1,j_p}-u^*_{i_p,j_p})^2 + (u^{*}_{i_p,j_p+1}-u^*_{i_p,j_p})^2}{2(u^*_{i_p,j_p})^2} \bigg) \nonumber \\
    &+ \frac{(m+1)}{h^2} \frac{u^{*}_{i_{m+1}+1,j_{m+1}}-u^*_{i_{m+1},j_{m+1}}}{u^*_{i_{m+1}, j_{m+1}}} -\frac{(m+1)}{h^2} \frac{u^{*}_{i_{m+1},j_{m+1}}-u^*_{i_{m+1}-1,j_{m+1}}}{u^{*}_{i_{m+1}-1,j_{m+1}}} \nonumber\\
    & + \frac{(m+1)}{h^2}\frac{u^{*}_{i_{m+1},j_{m+1}+1}-u^*_{i_{m+1},j_{m+1}}}{u^*_{i_{m+1}, j_{m+1}}} 
    - \frac{(m+1)}{h^2}\frac{u^{*}_{i_{m+1},j_{m+1}} -u^*_{i_{m+1},j_{m+1}-1}}{u^{*}_{i_{m+1},j_{m+1}-1}} \nonumber \\
    &+\frac{(m+1)}{h^2} \frac{(u^{*}_{i_{m+1}+1,j_{m+1}}-u^*_{i_{m+1},j_{m+1}})^2 + (u^{*}_{i_{m+1},j_{m+1}+1}-u^*_{i_{m+1},j_{m+1}})^2}{2(u^*_{i_{m+1},j_{m+1}})^2}. 
  \end{align} 
We use the fact that $u^*_{i_p}, p =i_0,\ldots i_{m-1}$ are minimum points, and the inequalities \eqref{eq:bdf1}-\eqref{eq:bdf2} as well as Lemma \ref{lm:phi} to deduce from the above equation that
\begin{align*}
    \frac{1}{h^2}\frac{d}{ds} & \mathcal{J}(u^*+sv) \bigg|_{s=0} \\
    \le& \frac{4(m+1)Ch^{-\frac12}LK}{\Delta t \min \hat{u}^nh^2} - \frac{q^2}{2h^2}+ \frac{2(m+1)N^4}{h^2}.
\end{align*}
  We take
  \begin{align*}
    q \ge \left(\frac{4(N^2+1)Ch^{-\frac12}LK}{\Delta t\min \hat{u}^n}\right)^{\frac{1}{2}} + 2\sqrt{N+1}N^2,
  \end{align*}
  and get
  \begin{align*}
     \frac{1}{h^2}\frac{d}{ds} & \mathcal{J}(u^*+sv) \bigg|_{s=0} \le -\frac{4(N^2-m)Ch^{-\frac52}LK}{\Delta t\min \hat{u}^n}-\frac{2(N^2-m)N^2}{h^2}<0.
  \end{align*}
  This contradicts to the assumption that $u^*$ is a minimizer. 
  Therefore, none of the assumptions \ref{i12} and \ref{i22} hold for 
  \[q=\max\left\{\frac{4N^2Ch^{-\frac12}LK}{\Delta t \min \hat{u}^n}+2N^6,\left(\frac{4(N^2+1)Ch^{-\frac12}LK}{\Delta t\min \hat{u}^n}\right)^{\frac{1}{2}} + 2\sqrt{N+1}N^2\right\}\]
  and 
 \begin{align*}
    \delta \le \delta_0 = \frac{K}{2h^2\sum_{p,m=0}^{N-1}(1+q)^{pm}}.
  \end{align*}
  Therefore, $u^*$ cannot touch the boundary of $\mathcal{A}_{h,\delta}$ for $\delta \ge \delta_0$ and Theorem \ref{thm} holds for the two dimensional case. Theorem \ref{thm:property+} then follows the same way as in the one dimensional case.

\section{Numerical examples}
 In this section, we present both one and two dimensional numerical examples to verify the mass conservation, positivity-preserving, energy dissipation properties,  and the convergence order for 
scheme \eqref{eq:fulld}.


\subsection{One dimension example}
First we consider a numerical example from \cite{BLS94}. We take $\mathbb{T}=[0,1]$ and the initial condition
\begin{align}\label{eq:uini}
  u_0(x) = \left(\varepsilon^{\frac{1}{2}} + \left[\frac{1+\cos 2\pi x}{2}\right]^m\right)^2.
\end{align}
Here $\varepsilon$ is taken to be $0.001$ and $m=1$ and $8$. Plots are made for 
$t=0,8\times 10^{-6},3.2\times 10^{-5},1\times 10^{-4}$ and $7.2 \times 10^{-4}$ 
with solid, dashed, dotted, dash-dot and dash-dash lines, respectively.
The solution $u$ is plotted in Figure \ref{fig:u1d}.
\begin{figure}
\begin{center}
\begin{tabular}{cc}
    \includegraphics[width=0.5\textwidth]{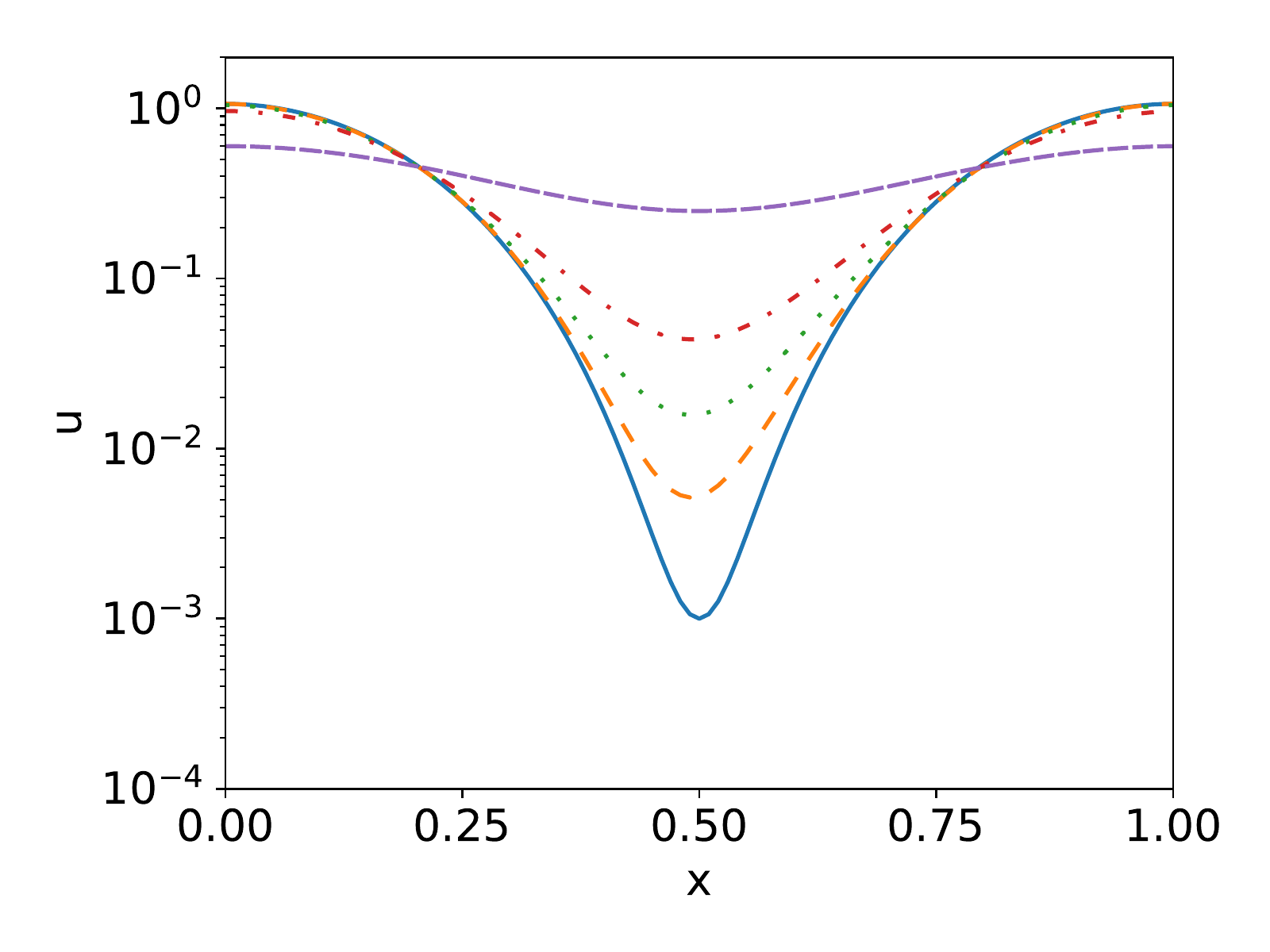}
&    \includegraphics[width=0.5\textwidth]{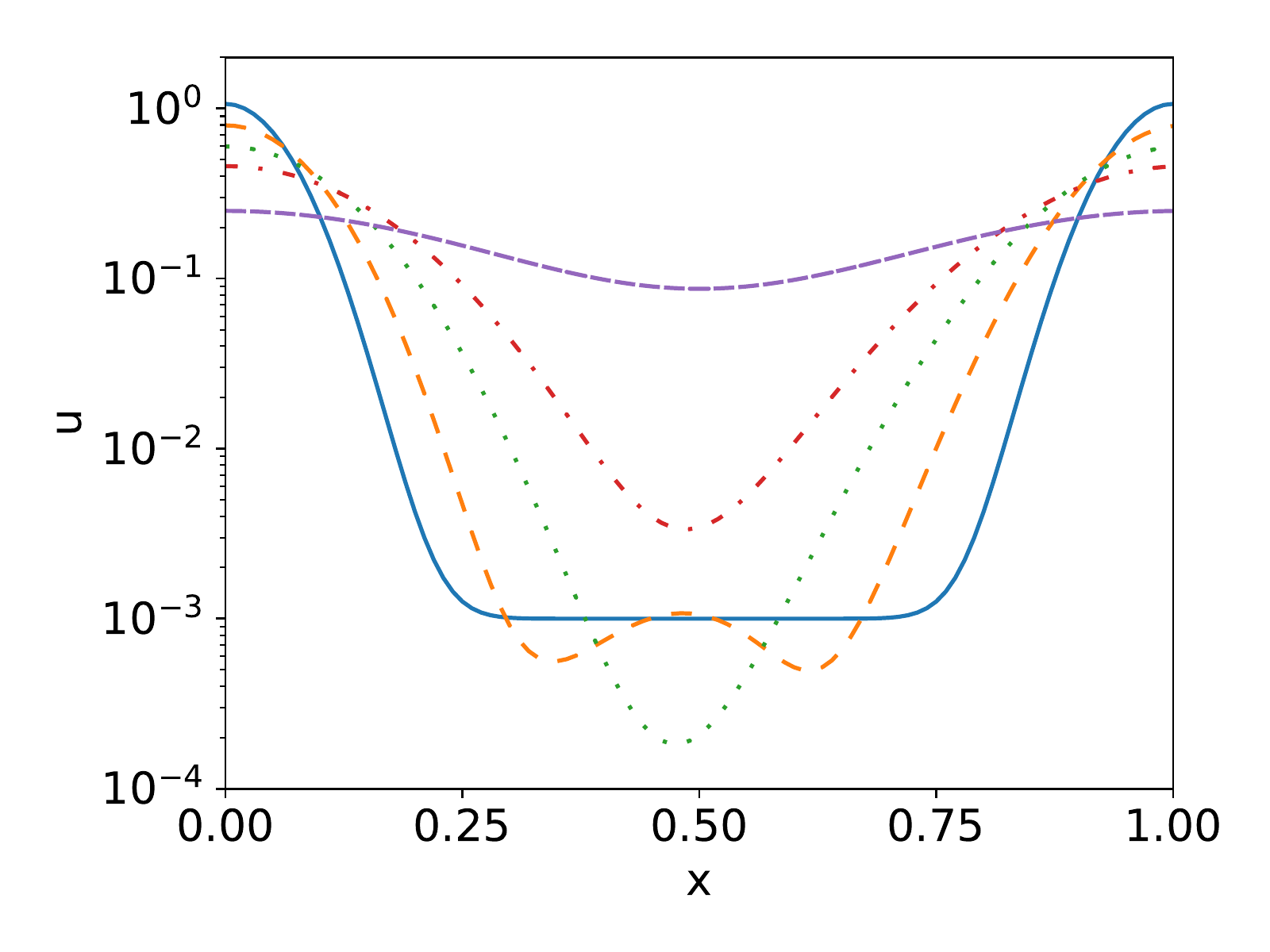}\\
\small $m=1$ & $m=8$
\end{tabular}
\end{center}
\caption{Numerical solutions to the equation \eqref{eq:DLSSnd} in $1$D}\label{fig:u1d}
\end{figure}
Mass variations, energy and minimum values over time are plotted in Figure \ref{fig:mem1d}.
These three figures verify the mass conservation, energy stability, and positivity of solutions, respectively.
\begin{figure}
\begin{center}
\begin{tabular}{ccc}
    \includegraphics[width=0.33\textwidth]{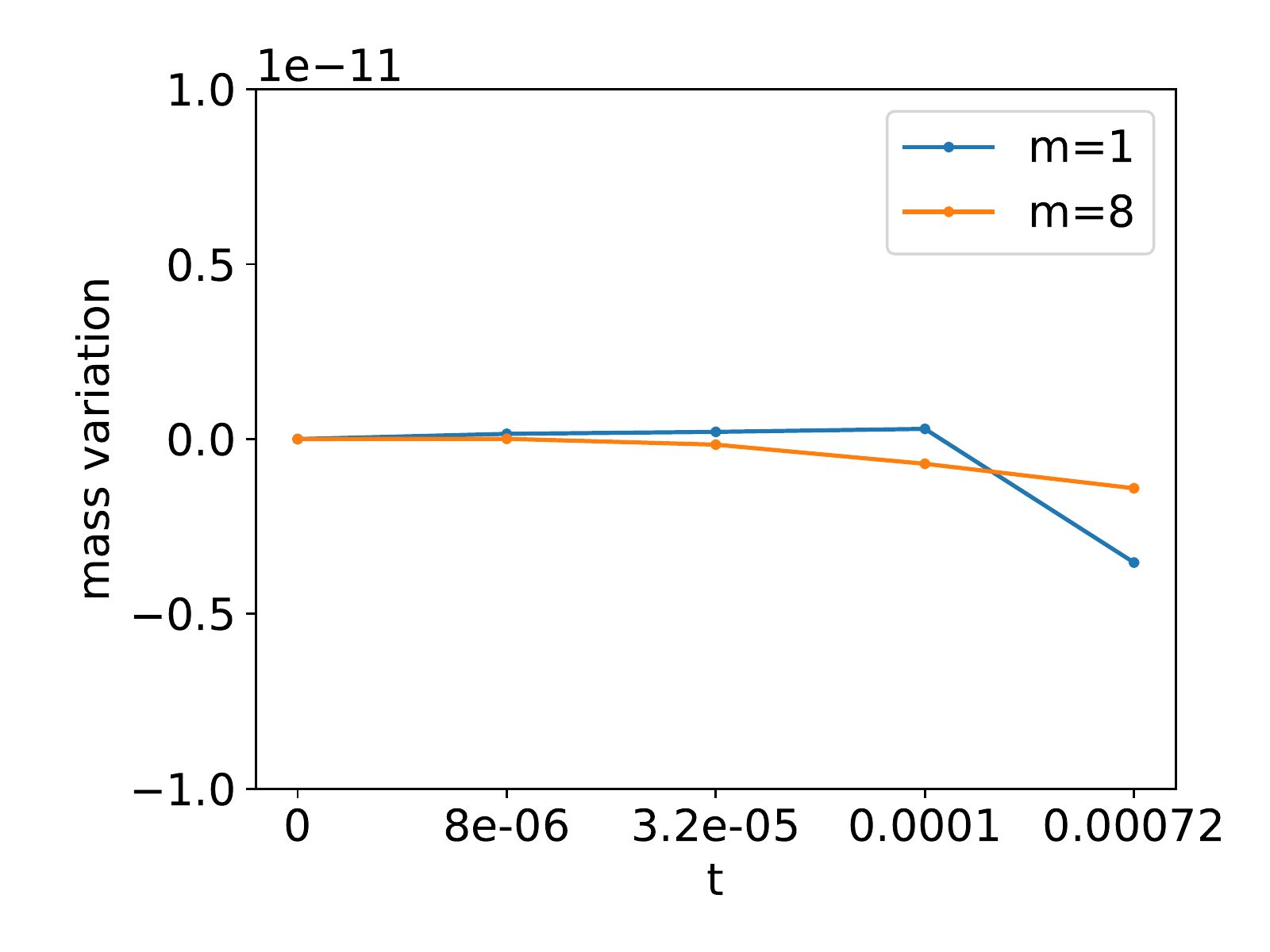}
&    \includegraphics[width=0.33\textwidth]{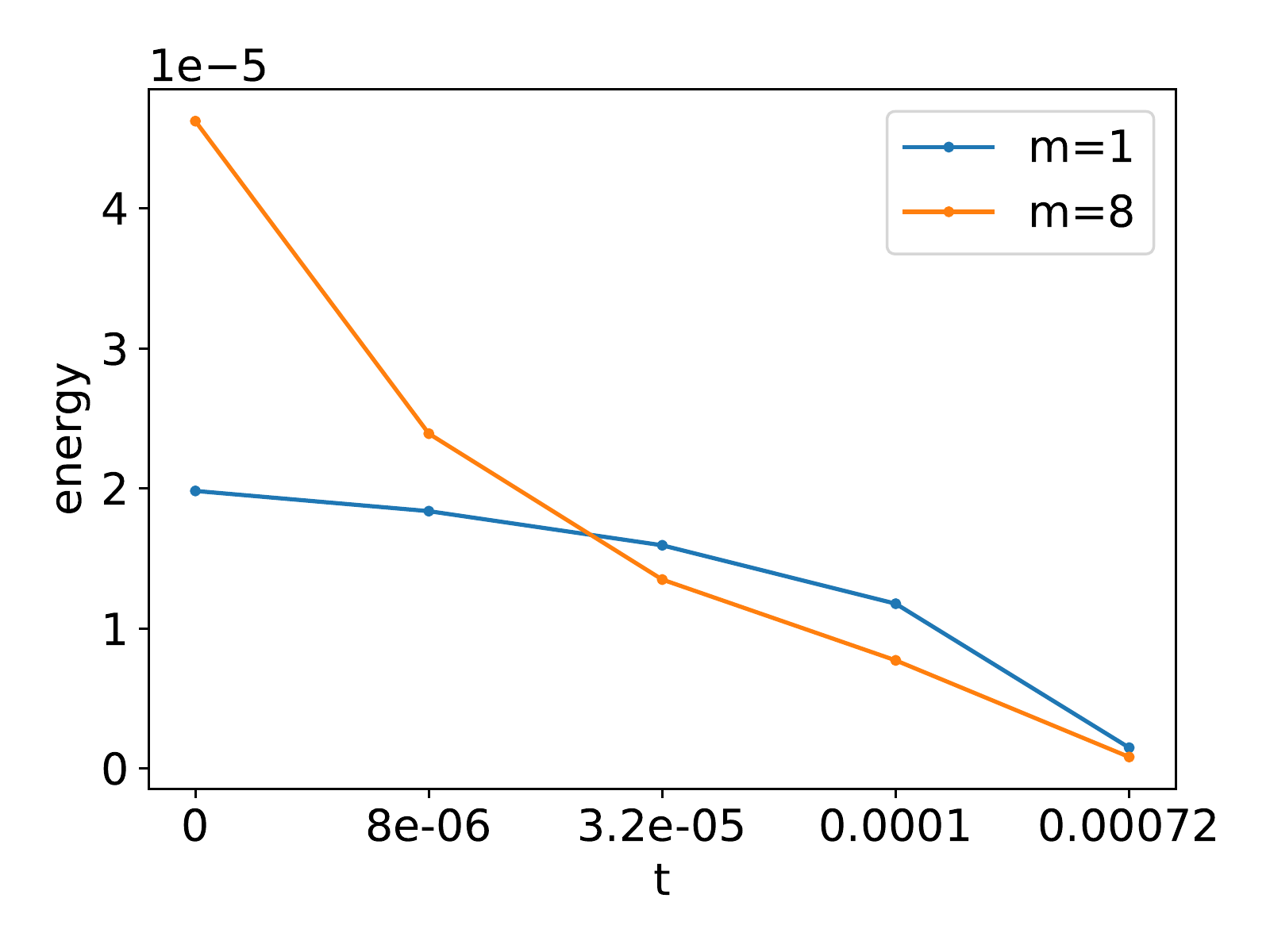} &  \includegraphics[width=0.33\textwidth]{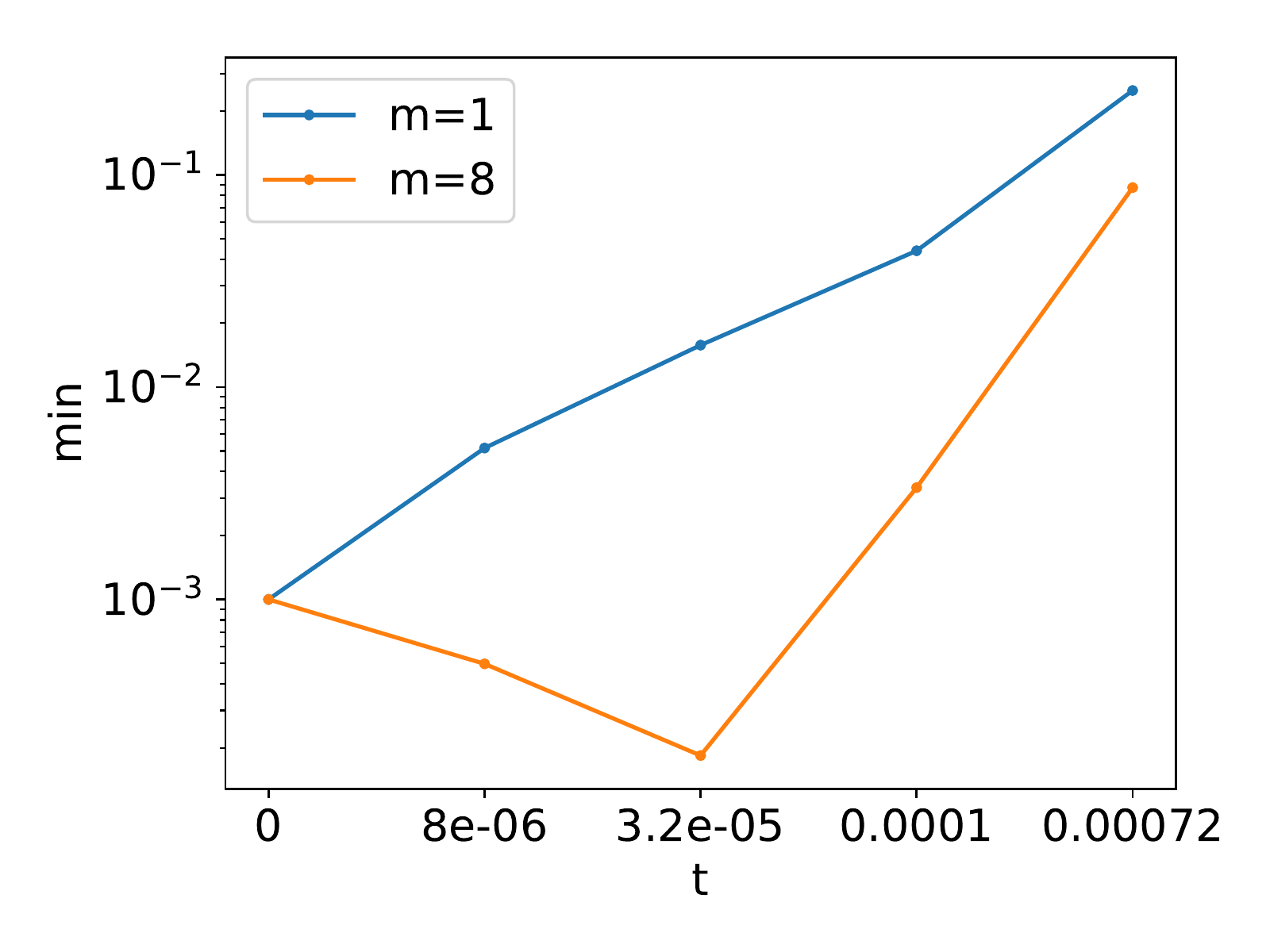}
\end{tabular}
\end{center}
\caption{Mass variations, energy and minimum values of $u$}\label{fig:mem1d}
\end{figure}
It can be seen that
 the minimum value remains positive. 

We also preform error computations for different mesh sizes. We take $\Delta t = 1.6\times 10^{-8} h$ and take the number of mesh intervals to be $10,20,40,80,160$, with mesh size $h=0.1,0.05,0.025,0.0125,0.00625$, respectively. We compare the solutions at time $t=7.2 \times 10^{-4}$. Since we do not have an exact solution, we take the result $\widetilde{u}$ ($h=0.00625$) as a reference solution. We compute the errors based on the fact that for a numerical scheme of $p$-th order, 
\begin{align*}
   \frac{u_{h} - \widetilde{u}}{u_{{h}/{2}}-\widetilde{u}} = 2^p + O(h).
\end{align*}
We interpolate $u_{h}$ using the nearest neighbor method in space to find its difference with $\widetilde{u}$ and calculate the $L^2$ error.
 The result is plotted in Figure \ref{fig:err1d}. The solid line is the error with respect to the reference solution and the dashed line is the line $1/h$. Thus the numerical scheme is approximately of order $1$ in $h$. 

\begin{figure}
\begin{center}
    \includegraphics[width=0.5\textwidth]{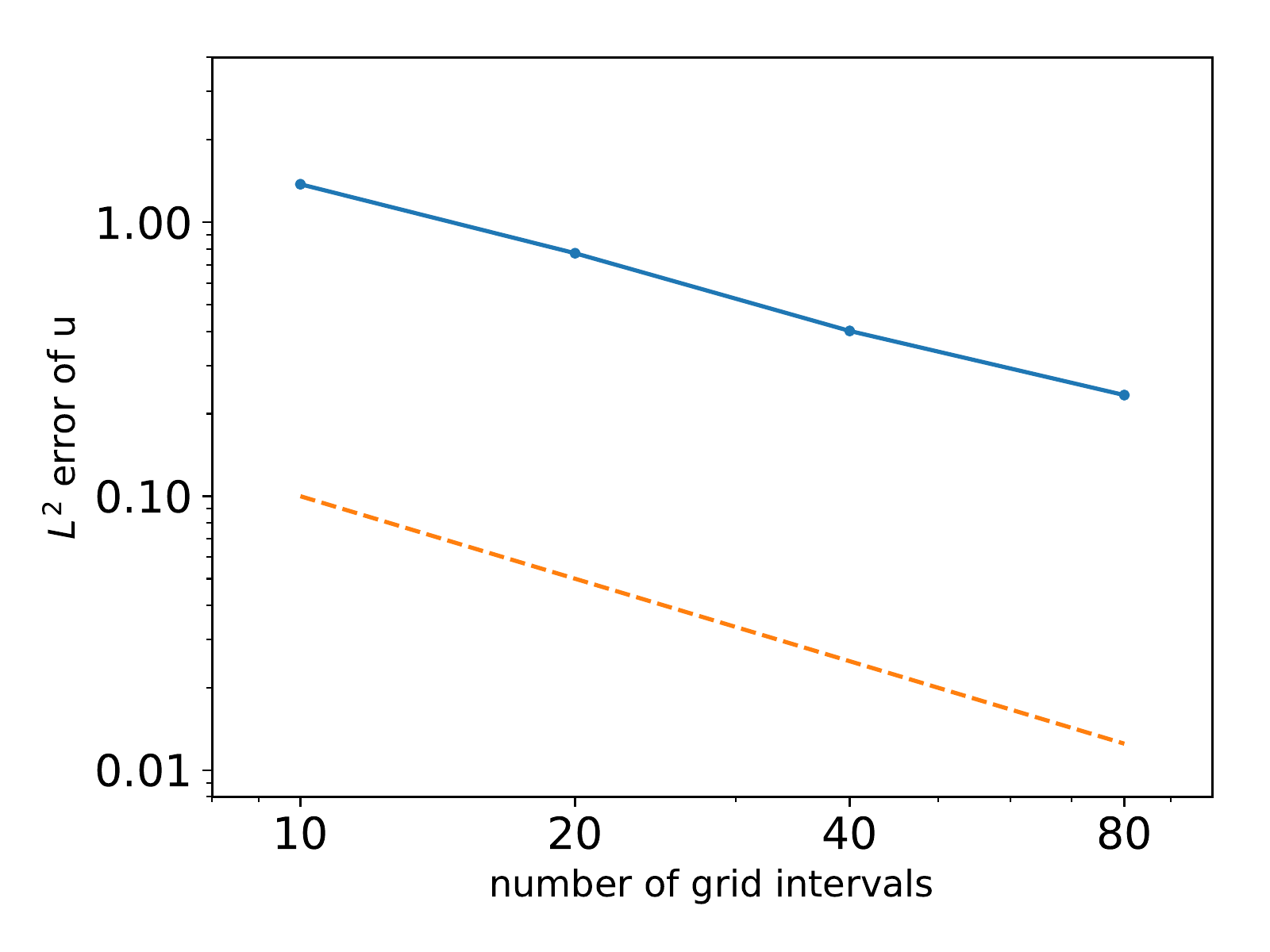}
\end{center}
\caption{Error and convergence order} \label{fig:err1d}
\end{figure}

\subsection{Two dimension example}
We compute the 2D problem on the torus $\mathbb{T}^2 = [0,1] \times [0,1]$ with the following initial condition
\begin{align*}
      u_0(x,y) = \left(0.001^{\frac{1}{2}} + \left[\frac{1+\cos 2\pi x \cdot \cos 2\pi y}{2}\right]^8\right)^2.
\end{align*}
This initial condition is chosen resembling the corresponding initial condition \eqref{eq:uini} in $1$D with $m=8$. However, there is one peak inside the domain $\mathbb{T}^2$ at point $x=y=\frac{1}{2}$. So the diffusion of this inner point and the four peaks at the boundary dominate the behavior of the solutions.

\begin{figure}
\begin{centering}
\begin{tabular}{ccc}
    \includegraphics[width=0.33\textwidth]{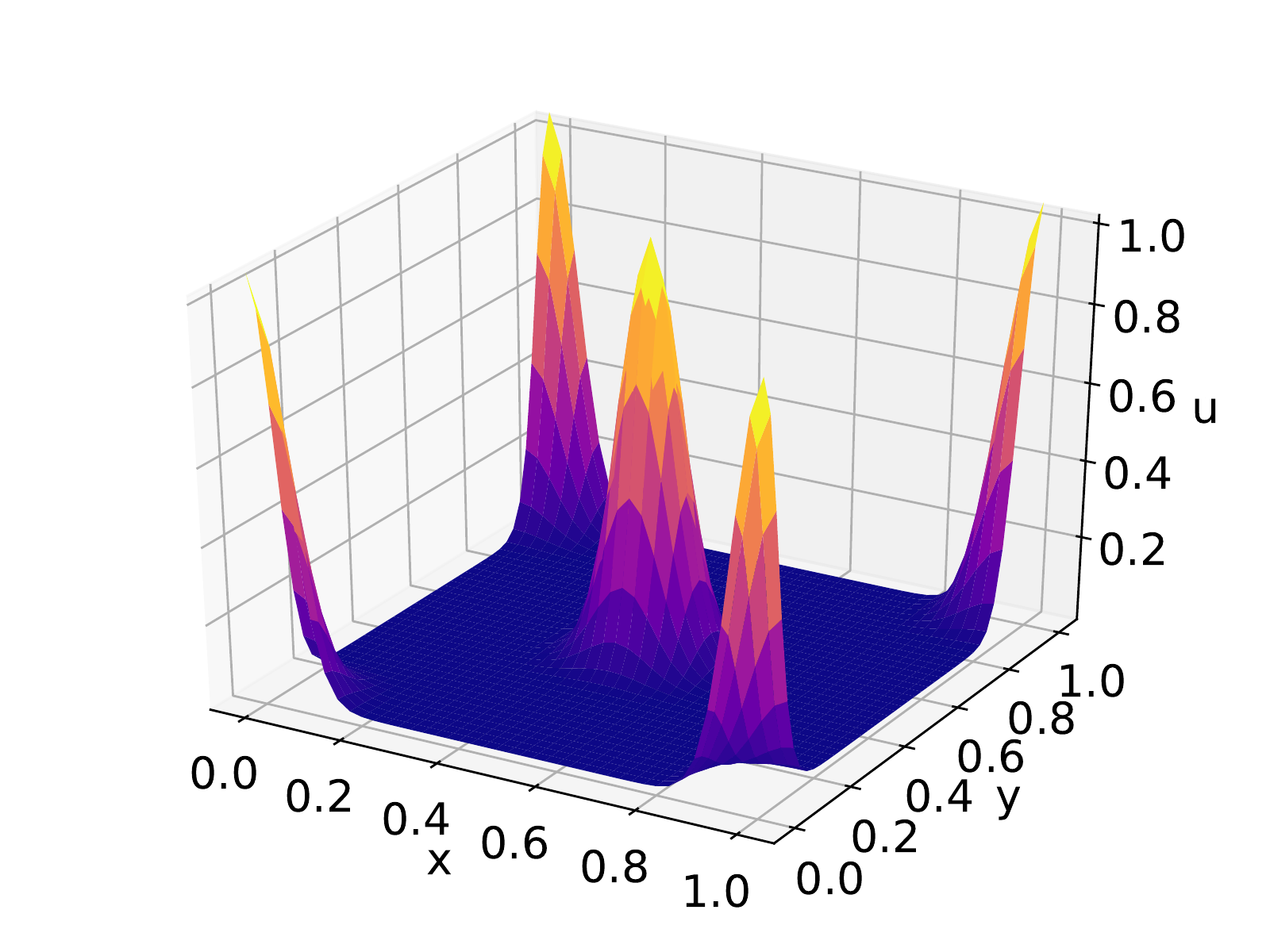}
&    \includegraphics[width=0.33\textwidth]{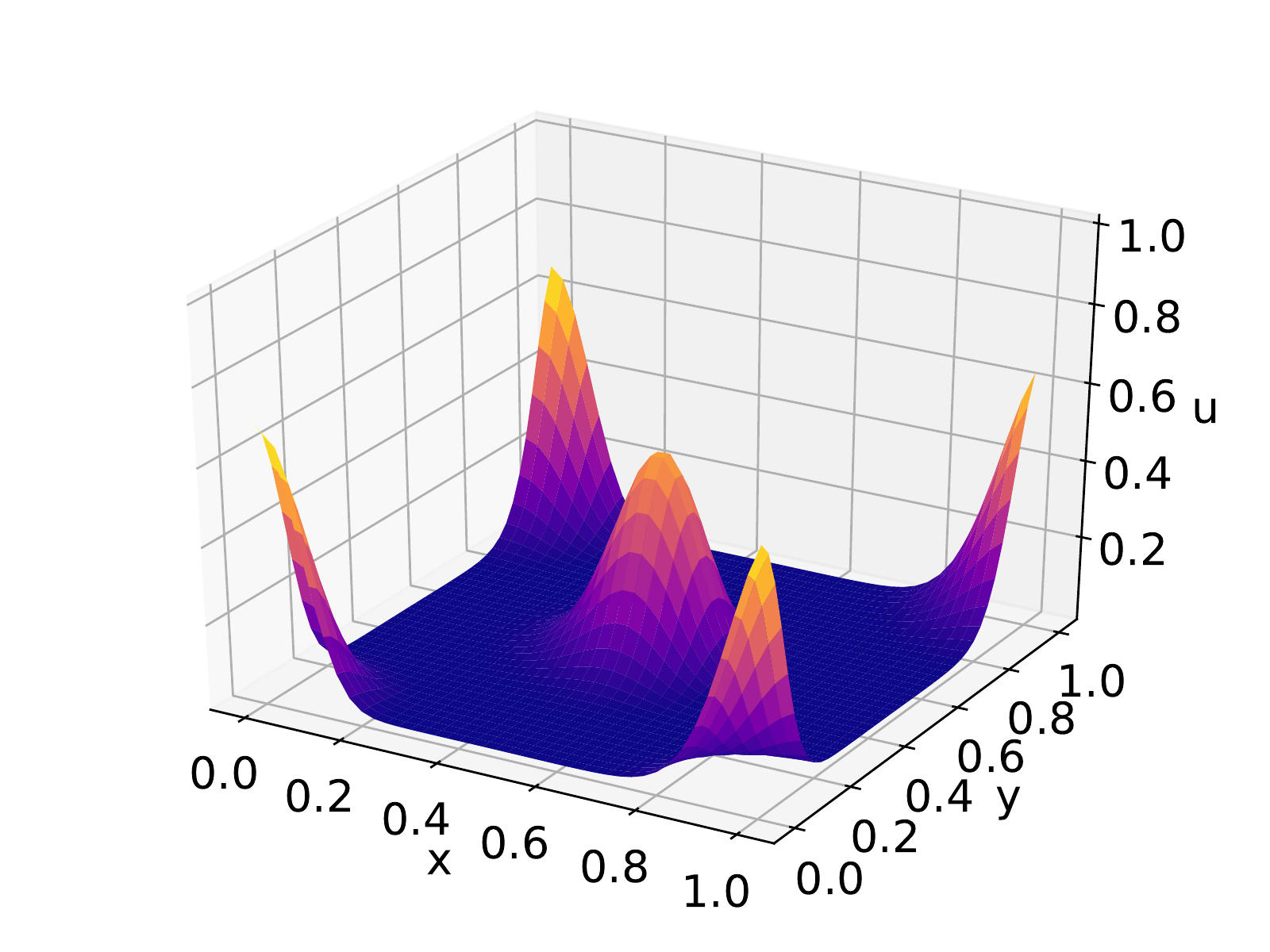} &\includegraphics[width=0.33\textwidth]{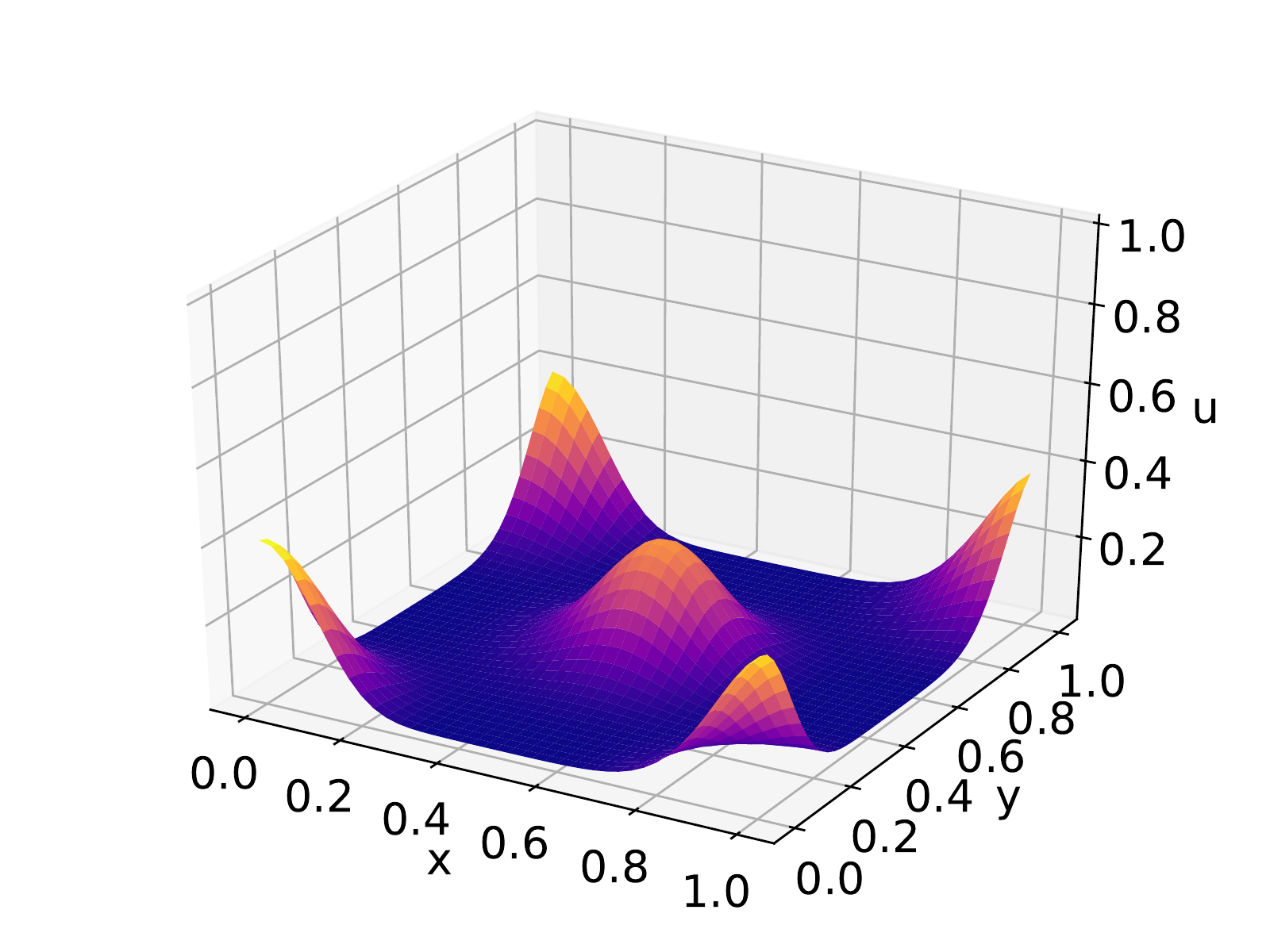}\\
\small $t=0.0$ & $t=8\times 10^{-6}$ & $t = 3.2 \times 10^{-5}$
\end{tabular}
\end{centering}
\begin{centering}
\begin{tabular}{ccc}
\includegraphics[width=0.33\textwidth]{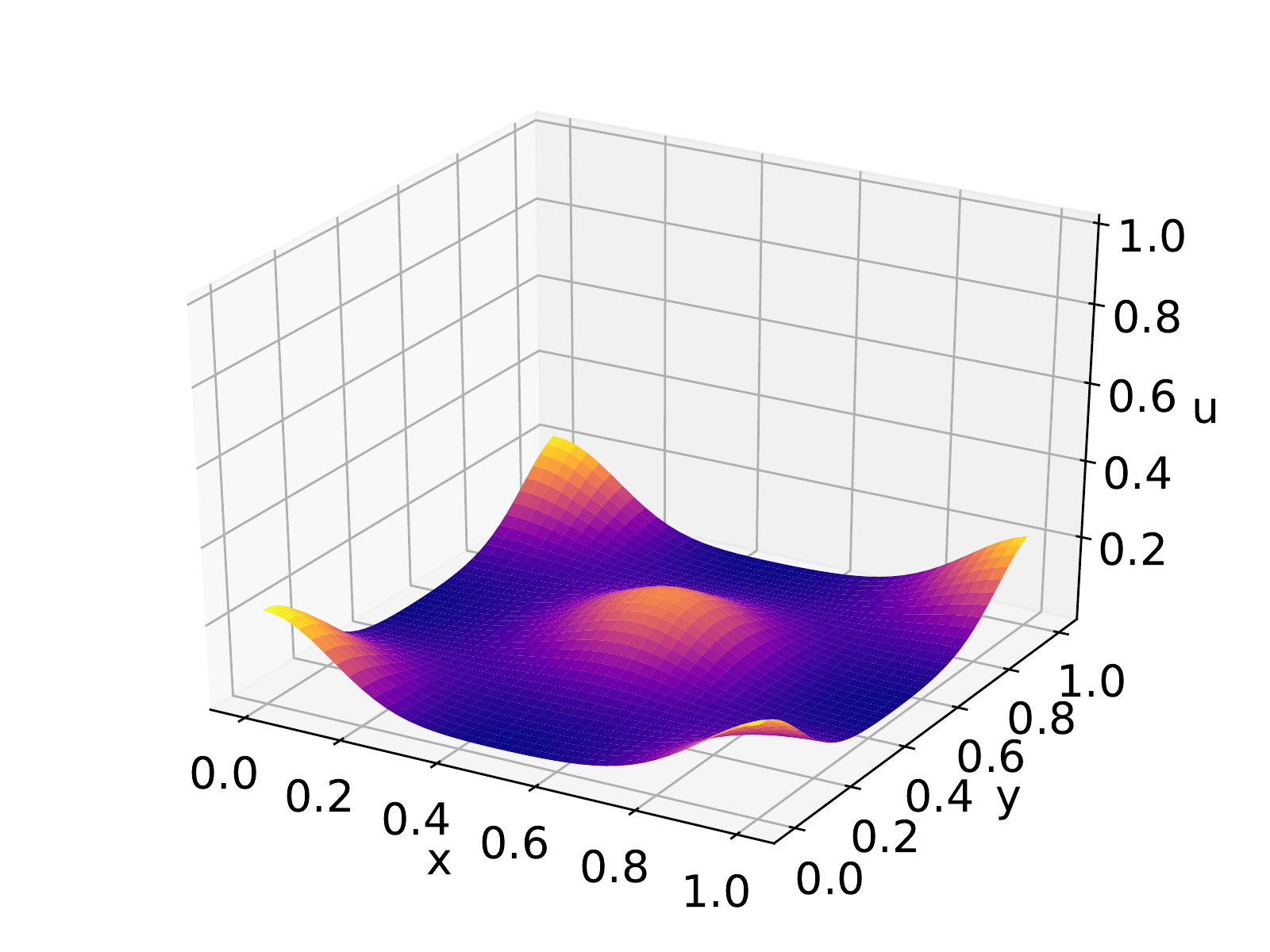}  
& \includegraphics[width=0.33\textwidth]{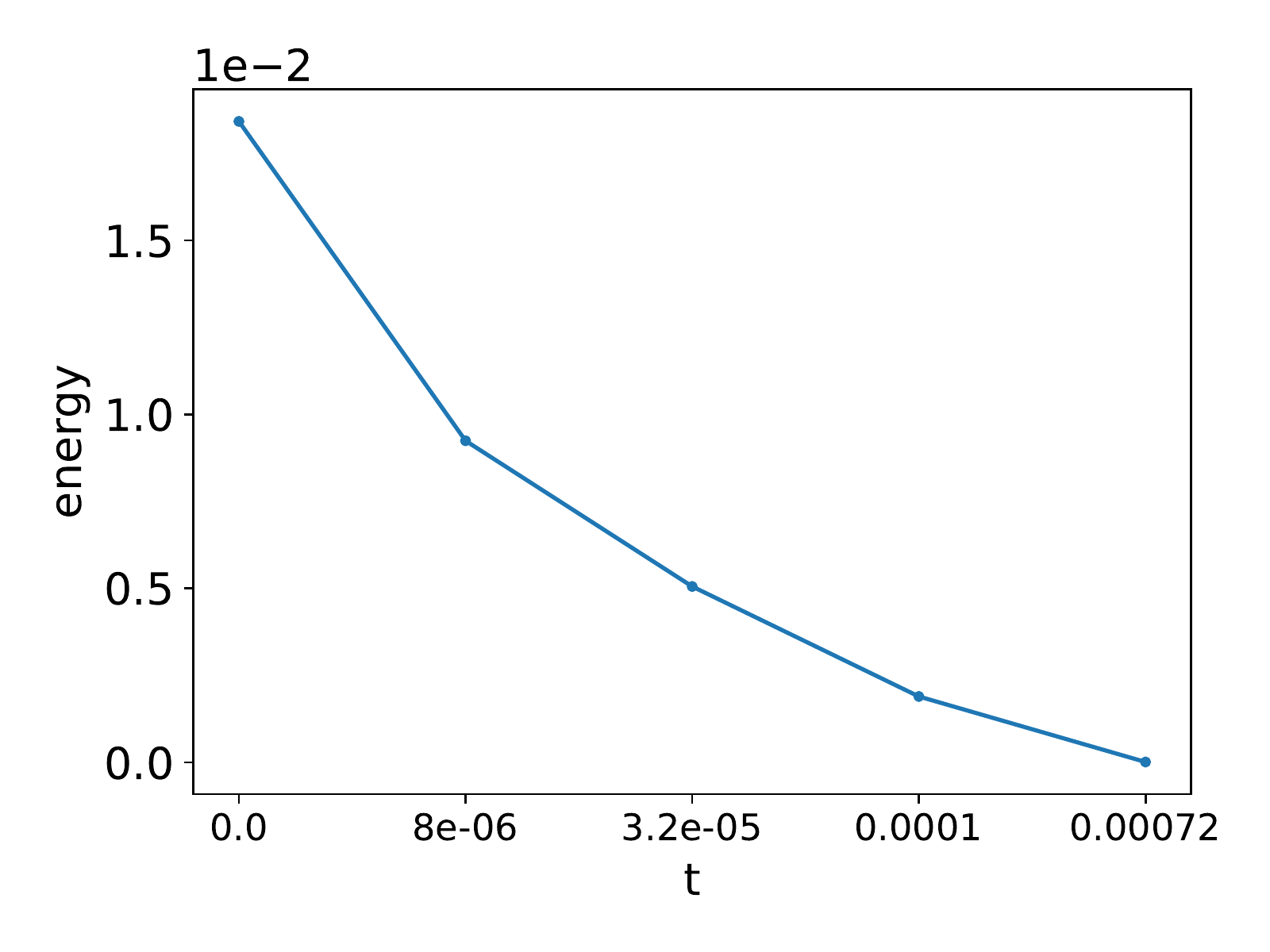} & \includegraphics[width=0.33\textwidth]{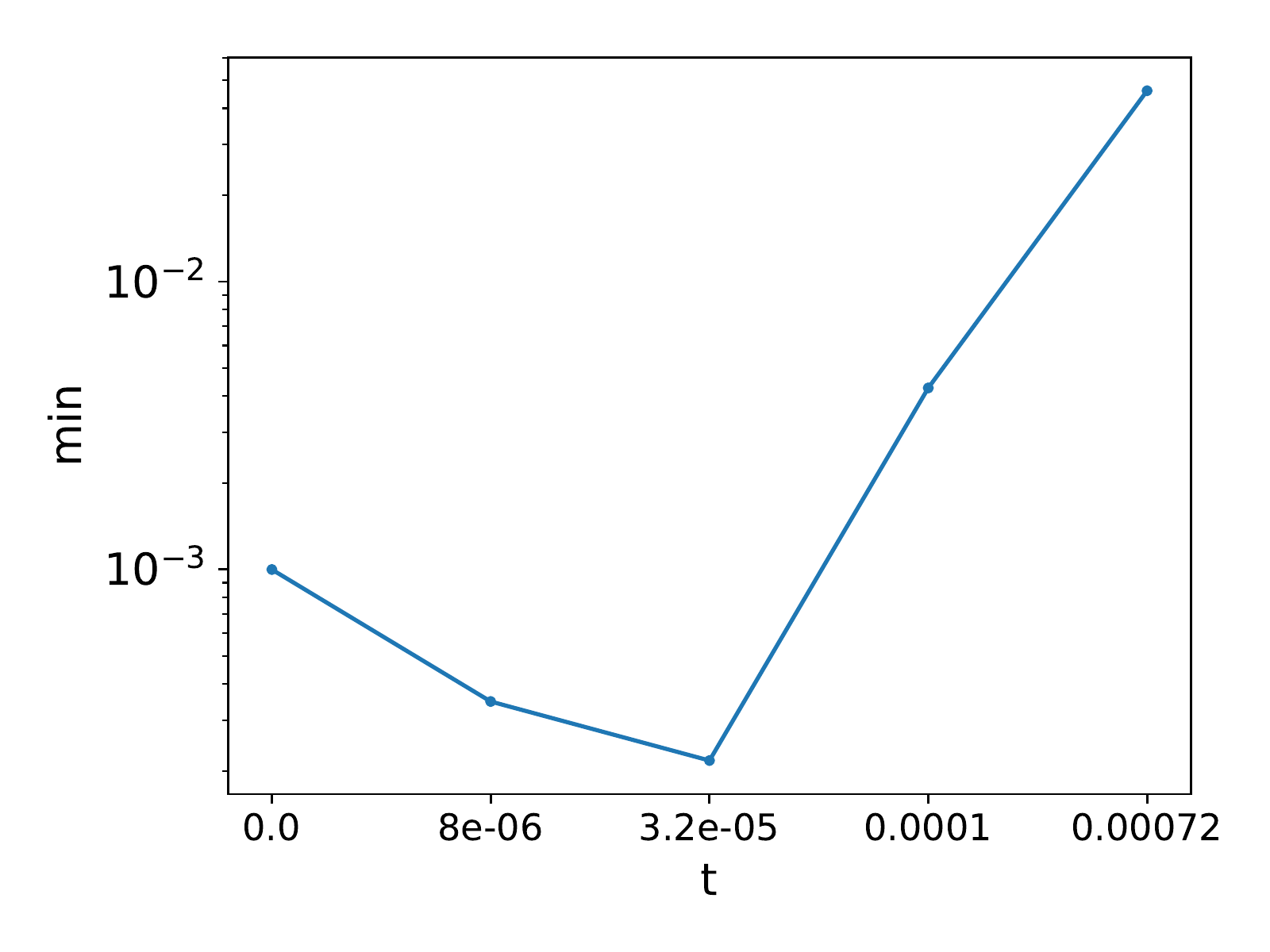}
 \\
 $t = 7.2 \times 10^{-4}$ & Energy & Minimum value
\end{tabular}
\end{centering}
\caption{Numerical results to the equation \eqref{eq:DLSSnd} in $2$D}\label{fig:m8}
\end{figure}

The solutions and their energy functions as well as minimum values are plotted over time in Figure \ref{fig:m8}. The minimum value decreases first and then increases over time but remains positive.
We also perform error computations for different mesh sizes in $2$D. The result is plotted in Figure \ref{fig:err2d}. It is similar with the $1$D example and the numerical convergence order is approximately one in $h$.
\begin{figure}
\begin{center}
\begin{tabular}{c}
    \includegraphics[width=0.5\textwidth]{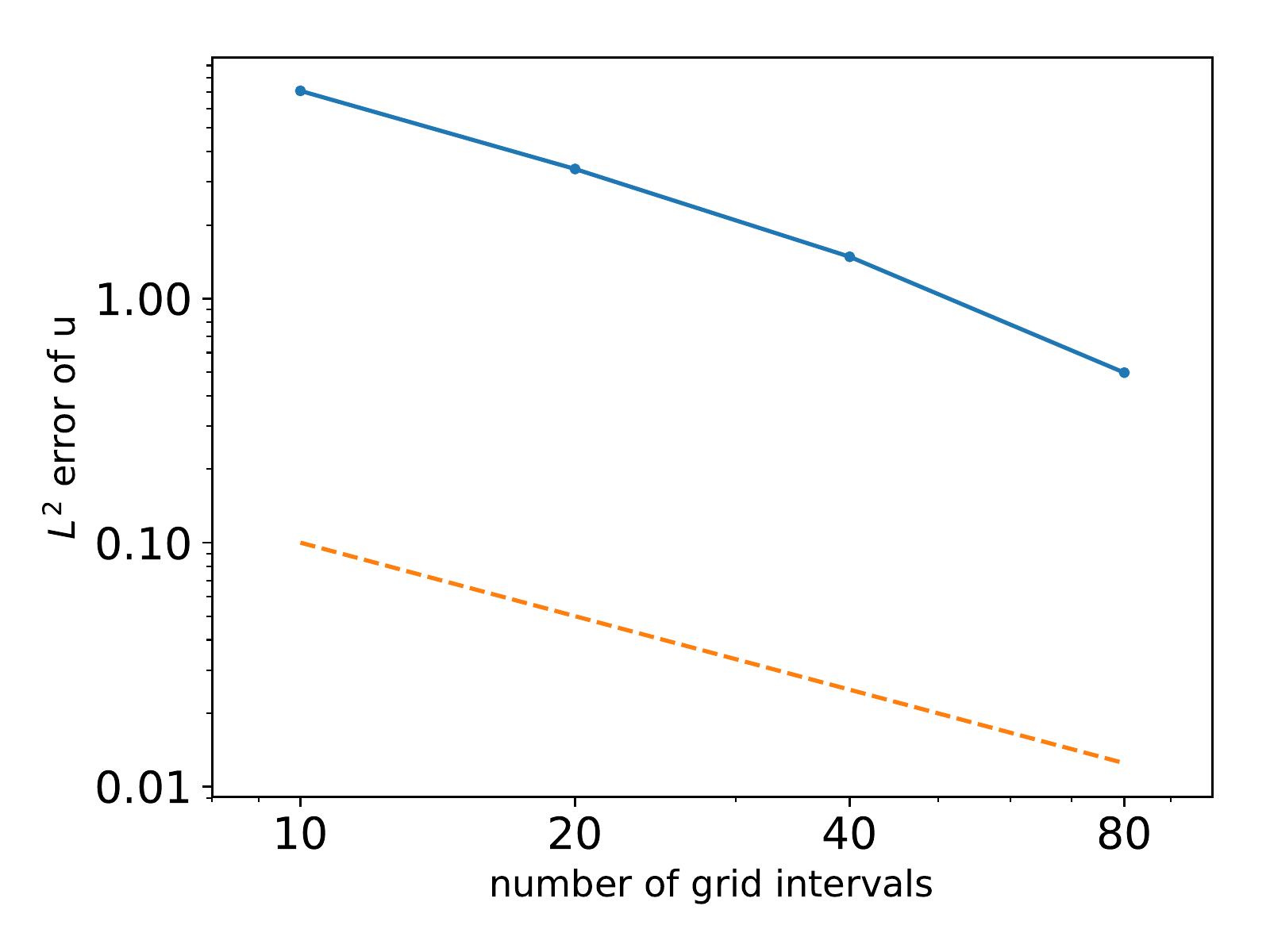}
\end{tabular}
\end{center}
\caption{Error and convergence order in $2$D} \label{fig:err2d}
\end{figure}

 \section*{Acknowledgement} The first author is funded by KAUST. The second author is funded by the National Science Foundation under Grant DMS1812666.
 The authors are grateful to Athanasios E. Tzavaras for valuable suggestions 
 and comments. 


\bibliographystyle{siamplain}

\end{document}